\documentclass[aos, preprint]{imsart}

\RequirePackage{amsthm,amsmath,amsfonts,amssymb}
\RequirePackage[numbers]{natbib}
\RequirePackage[colorlinks,citecolor=blue,urlcolor=blue]{hyperref}
\RequirePackage{graphicx}
\usepackage{algorithm, algorithmic}
\usepackage{enumitem}
\usepackage{tikz}
\usetikzlibrary{shapes.multipart,patterns,arrows}

\usepackage{nicematrix}
\NiceMatrixOptions
{
	custom-line ={command= H, tikz= dashed, width= 1mm}, 
	custom-line = {letter= I, tikz= dashed, width= 1mm}, 
}

\startlocaldefs

\numberwithin{equation}{section}
\theoremstyle{plain}
\newtheorem{theorem}{Theorem}
\numberwithin{theorem}{section}
\newtheorem{corollary}[theorem]{Corollary}
\theoremstyle{definition}

\newtheorem{prop}[theorem]{Proposition}

\newtheorem{definition}[theorem]{Definition}

\newtheorem{remark}{Remark}

\def\param{\boldsymbol{{\mathbf{\theta}}}}
\def\Param{\boldsymbol{\Theta}}

\def\blambda{\boldsymbol{\lambda}}

\def\v{\mathbf{v}}
\def\w{\mathbf{w}}

\def\b{\mathbf{b}}
\def\d{\mathbf{d}}

\def\I{\mathbf{I}}
\def\b{\mathbf{b}}

\def\blambda{\boldsymbol{\lambda}}

\def\param{\boldsymbol{\theta}}
\def\Param{\boldsymbol{\Theta}}

\def\<{\langle}
\def\>{\rangle}

\def\A{\mathbf{A}}
\def\B{\mathbf{B}}

\def\0{{\mathbf 0}}

\def\.{\hskip.06cm}

\def\P{{\textup{\textsf{P}}}}

\DeclareMathOperator{\Var}{\textnormal{Var}}

\def\R{\mathbb{R}}

\def\E{\mathbb{E}}
\def\P{\mathbb{P}}

\def\X{\mathbf{X}}

\def\b{\mathbf{b}}

\def\x{\mathbf{x}}

\def\w{\mathbf{w}}

\DeclareMathOperator*{\argmin}{arg\,min}

\endlocaldefs

\usepackage{xcolor}

\begin{document}
	
	\begin{frontmatter}
		\title{Differentially Private Two-Stage \\ Empirical Risk Minimization \\  with Applications to Individualized Treatment Rule}
		\runtitle{Differentially Private Two-Stage Empirical Risk Minimization}
		
		\begin{aug}
			\author[A]{\fnms{Joowon}~\snm{Lee}\ead[label=e1]{jlee2256@wisc.edu}}
			\and
			\author[A,B]{\fnms{Guanhua}~\snm{Chen}\ead[label=e2]{gchen25@wisc.edu}}
			\address[A]{Department of Statistics, University of Wisconsin - Madison
				\printead[presep={,\ }]{e1}}
			
			\address[B]{Department of Biostatistics and Medical Informatics, University of Wisconsin - Madison
				\printead[presep={,\ }]{e2}}
		\end{aug}
		
		\begin{abstract}
			Differential privacy provides a formal framework for releasing statistical estimators that limit how much any single observation can influence the output, by injecting calibrated random noise. We study differentially private estimation in two-stage procedures common in causal inference and individualized treatment rule (ITR) learning, in which data-dependent weights are first estimated to enforce covariate balance and a parameter of interest is then obtained by weighted empirical risk minimization. We propose Differentially Private Two-Stage Empirical Risk Minimization (DP-2ERM), which privatizes the final estimator directly through objective perturbation calibrated to the data-dependent sensitivity of the full pipeline. The analysis combines deterministic weight-perturbation bounds for several covariate-balancing methods (inverse propensity weighting, entropy balancing weighting, and maximum mean discrepancy weighting) with probabilistic sensitivity bounds for the second-stage solution. The resulting calibration is sharper than the natural stage-wise composition baseline, which the same sensitivity analysis supplies as a byproduct. Simulation studies and a benchmark application to ITR learning demonstrate the improved privacy--utility trade-off.
		\end{abstract}
		
		
	\end{frontmatter}

	\section{Introduction}
	\label{Introduction}
	
	\subsection{Differential Privacy and Weighted Empirical Risk Minimization} 
	
	Differential privacy (DP) provides a rigorous notion of data protection by requiring that the
	distribution of any released output changes only slightly when the dataset differs in a single entry (Definition~\ref{def:DP}; see, e.g., \citealp{dwork2006calibrating,dwork2014algorithmic}).
	DP mechanisms achieve this by injecting calibrated random noise, with magnitude governed by a
	privacy budget: smaller privacy parameters correspond to stronger protection but typically require
	larger perturbations and hence reduced statistical utility.
	This framework is indispensable in sensitive domains such as healthcare and finance.
	Foundational approaches to private learning for convex \textit{empirical risk minimization} (ERM)  include
	objective perturbation \citep{chaudhuri2011differentially} (see also \citealp{kifer2012private}) and differentially private
	stochastic gradient methods such as DP-SGD \cite{abadi2016deep}. 
	
	Standard DP methods, including DP-ERM, effectively privatize single-stage estimation procedures, meaning that the input data is used once in the whole pipeline by solving the associated optimization problem. However, modern statistical pipelines often involve more complex, multi-stage architectures. For instance, in the ERM framework, sample-specific weights are introduced to form a weighted average of the empirical objectives to account for data heterogeneity, confounding, or selection bias.
	We formalize this class of problems as \textit{two-stage empirical risk minimization} (2ERM), which utilizes the input data in two distinct phases:
	\begin{description}[leftmargin=1.5cm, itemsep=0.1cm]
		\item[Stage 1:] (\textbf{Find optimal sample weights}) Use the data to compute sample-specific weights that satisfy a specific balancing criterion. 
		\item[Stage 2:] (\textbf{Find optimal parameter}) Obtain the estimated parameter of interest by solving a weighted ERM problem with the weights obtained in the first stage. 
	\end{description}
	See Section \ref{sec:problem_formulation} for a precise problem formulation of 2ERM. Note that the data-dependent weights found in the first stage serve as ``nuisance parameters'': they are instrumental for achieving various desired properties of the parameter of interest, yet they are not the primary objects of inference.
	
	A paramount instance of 2ERM—and the primary motivation for this work—is the estimation of \textit{individualized treatment rules} (ITRs) in precision medicine. The goal of ITR is to recommend personalized treatments by leveraging individual covariates to maximize clinical benefit \cite{kosorok2019precision,zhao2012estimating,zhang2012robust,qian2011performance,athey2021policy,wallace2015doubly}. In observational studies, where treatment assignment is not random, this process is inherently two-stage. The first stage must estimate individual-specific weights—typically \textit{inverse probability weighting} (IPW) \cite{rosenbaum1983central}, \textit{maximum mean discrepancy} (MMD) \cite{chen2024robust}, or \textit{entropy balancing weight} (EBW) \cite{hainmueller2012entropy}—to balance covariate distributions between treatment groups and address confounding bias. Then, in the second stage, the optimal treatment rule can be effectively estimated by solving the associated weighted ERM \cite{athey2021policy,tian2014simple,chen2017general}. 
	
	Beyond ITR, the 2ERM structure encapsulates a wide array of problems in statistics and machine learning. In domain adaptation/transfer learning \cite{zhuang2020comprehensive}, importance weighting is the standard remedy for covariate shift, where the first stage estimates density ratios—often via Kernel Mean Matching \cite{gretton2009covariate} or KL-divergence minimization \cite{shimodaira2000improving}—to align source and target distributions. Similarly, in robust learning with noisy labels, the first stage estimates ``cleanliness'' weights \cite{liu2015classification} to down-weight corrupted samples. The framework is equally relevant to algorithmic fairness, where pre-processing techniques reweight datasets to enforce statistical independence between sensitive attributes and labels \cite{kamiran2012data, calmon2017optimized} prior to classification.

	\subsection{Differentially Private Two-Stage ERM} 
	
	The goal of this work is to develop a rigorous and effective framework for \textit{differentially private two-stage ERM} (DP-2ERM), applicable to convex 2ERM pipelines, illustrated here on ITR. Our main application is ITR estimation, and the paper carries out its full theoretical and empirical development in the convex balancing-weight regime most relevant to this setting, while also providing a general analysis blueprint for sequential weighting-then-ERM procedures.  Existing work has studied DP in related settings, including DP \textit{outcome-weighted learning} (OWL) with known propensity weights \cite{spicker2024differentially}, DP weighted ERM under an asymptotic stability assumption that effectively reduces the two-stage sensitivity to that of a single-stage problem \cite{giddens2023differentially}, and DP covariate balancing for scalar \textit{average treatment effect} (ATE) estimation where the weights are released \cite{ohnishi2024differentially}. However, to the best of our knowledge, none provides end-to-end DP guarantees for ITR learning when the balancing weights are data-dependent solutions to a convex program and are not released; see Section~\ref{sec:related_work_DP_CI} for further discussion. Given the sensitivity of patient medical histories, the ability to estimate effective ITRs in a privacy-preserving manner is of substantial practical importance. Developing such a DP-2ERM framework, however, poses several challenges.
	
	First, a natural way to privatize 2ERM is to privatize each stage separately and then invoke a composition theorem. However, applying this strategy directly in our setting is nontrivial, because it would require the second-stage weighted ERM mechanism to satisfy a DP guarantee uniformly over all possible first-stage outputs; see, for example, \cite[Theorem 3.2]{dong2022gaussian} and \cite[Theorem 3.20]{dwork2014algorithmic}. In 2ERM, this uniformity requirement is particularly demanding because the first-stage weights are data-dependent and can vary substantially across neighboring datasets.
	
	Second, calibrating the second stage against such worst-case first-stage outputs can be overly conservative. In extreme cases, the learned weights may place substantial mass on a small subset of observations, making the weighted ERM objective highly sensitive to individual records. Protecting against all such possibilities uniformly would therefore require substantially more noise, leading to a poor privacy--utility trade-off and potentially a substantial loss of statistical utility in the final estimator.
	
	Third, a literal stage-wise composition strategy would also require privatizing the first-stage weight computation itself \cite{niu2022differentially}. However, adding noise directly to the weights can disrupt the constraints they are designed to enforce. For example, in causal inference, weights are constructed to balance covariate distributions between treatment groups in order to mitigate confounding. Even small perturbations for DP can lead to a failure of covariate balance, resulting in biased and unreliable treatment effect estimates.
	
	Lastly, allocating a significant portion of the privacy budget to the first stage is fundamentally inefficient. In the 2ERM paradigm, the sample-specific weights are internal nuisance quantities in the estimation pipeline and are not released. Using composition would therefore expend a limited privacy budget on nuisance parameters rather than on the primary object of inference.
	
	In our DP-2ERM framework, we instead establish end-to-end differential privacy for the full two-stage procedure directly. Specifically, we keep the first stage (weight computation) in its non-private form so that its balancing constraints remain intact, and then apply a tailored version of objective perturbation \cite{chaudhuri2011differentially} suitably extended to the 2ERM setting only to the released second-stage solution. Figure~\ref{fig:DP2ERM}
	summarizes this design, contrasting DP-2ERM with approaches that privatize the first-stage
	weights: the balancing weights remain internal, and objective perturbation is applied only to
	the released second-stage estimator.
	
	\begin{figure}[ht!]
		\centering
		\includegraphics[width=1\textwidth]{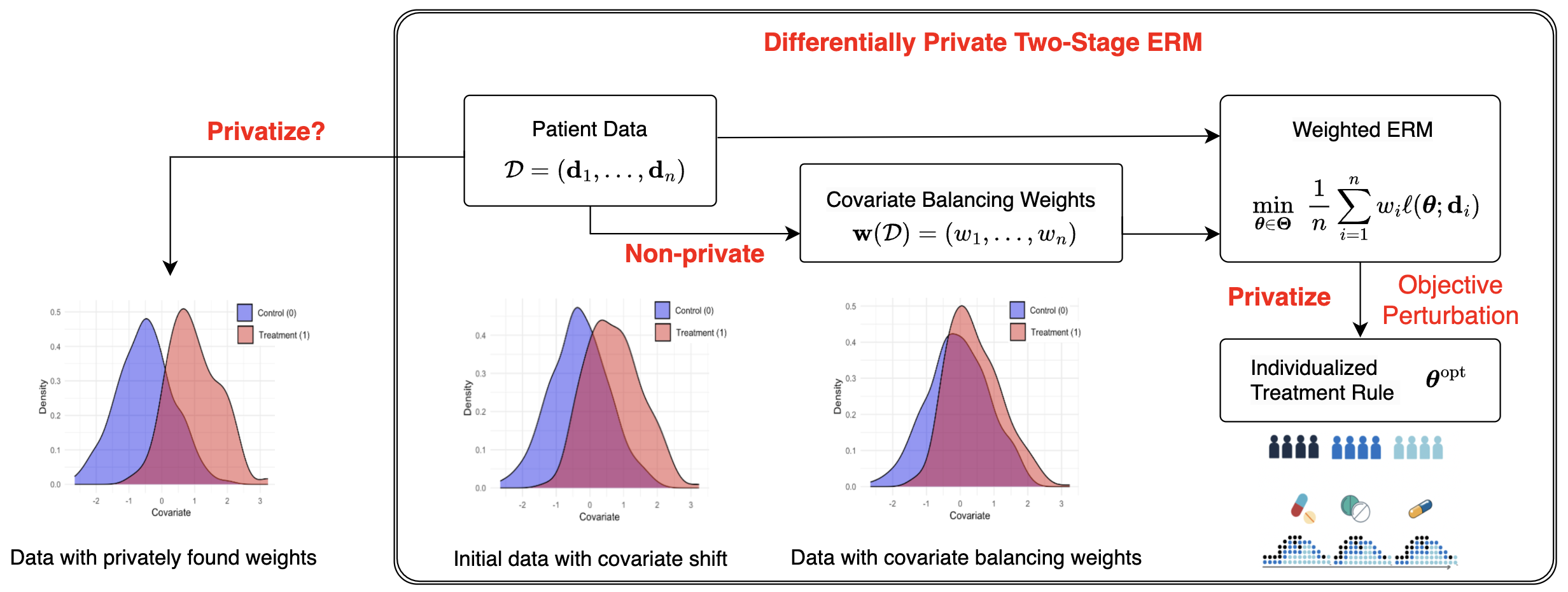}
		\caption{\textit{Differentially Private Two-Stage ERM Framework.} Our approach maintains non-private weight computation in the first stage to effectively address covariate shift, while applying privacy mechanisms exclusively to the second stage where weighted ERM is solved. This design is critical because standard composition theorems require calibrating noise against worst-case sequential dependencies to guarantee privacy, resulting in excessive noise injection that destroys utility. Our method ensures end-to-end differential privacy while preserving the effectiveness of the weights for balancing covariates.} 
		\label{fig:DP2ERM}
		\vspace{-0.1in}
	\end{figure}

	Crucially, 2ERM exhibits a unique amplification of data dependency: a change in a single data point shifts the weights for the entire sample, which in turn perturbs the objective function of the second stage. Because the objective function changes more aggressively than in a typical single-stage setting or a parallel stage problem, the final output is significantly more reactive to individual data points. The primary theoretical contribution of our direct approach is a sharp end-to-end analysis of this amplification effect. It replaces the uniform-in-weights calibration required by stage-wise composition (see Section~\ref{sec:related_work_composition}) with method-specific weight-stability bounds coupling the two stages, yielding a sharpened sufficient condition on the noise level and objective-perturbation regularization required to guarantee $(\varepsilon,\delta)$-DP of the released second-stage estimator.

	\subsection{Contributions}
	
	At a high level, this paper develops DP-2ERM, a differentially private framework
	for convex two-stage ERM pipelines, with ITR learning as the main application.
	Our main contributions are as follows.
	
	\begin{enumerate}[label=(\roman*), itemsep=0.1cm]
		
		\item \textbf{A sensitivity bound for objective perturbation with data-dependent weights
			(Theorem~\ref{thm:sensitivity_objective_perturbation}).}
		We prove an explicit sensitivity bound for objective perturbation in two-stage
		weighted ERM. In contrast to the standard single-stage setting, a one-record
		perturbation changes the second-stage objective both directly through the changed
		observation and indirectly through the upstream weights $\w(\mathcal D)$. Our
		theorem captures this indirect dependence through two weight-perturbation
		summaries, $W_1$ and $W_2$, and converts bounds on these quantities into
		sufficient calibrations for the noise level and regularization under both Gamma
		($\varepsilon$-DP) and Gaussian (($\varepsilon,\delta)$-DP) objective perturbation.
		The analysis also covers weighted losses with arbitrary-rank per-sample Hessians
		controlled by a trace bound, using a determinant perturbation inequality for
		weighted sums of positive semidefinite matrices.
		
		\item \textbf{Deterministic stability bounds for standard weighting schemes.} We establish non-asymptotic, deterministic stability bounds for the first-stage weight maps of IPW in randomized trials (Proposition~\ref{prop:stability_IPW_randomized_trial}), IPW with known and estimated logistic regression parameters (Propositions~\ref{prop:IPW_known_beta} and \ref{prop:IPW_estimated_stability}), MMD balancing (Proposition~\ref{prop:MMD_stability}), and entropy balancing (Theorem~\ref{thm:entropy_weight_stability}). Together with (i), these bounds yield method-specific DP guarantees; separately, they are of independent statistical interest as stability results for widely used covariate-balancing programs.
		
		\item \textbf{End-to-end privacy, utility, and consistency guarantees.} Plugging the stability results of (ii) into the sensitivity inequality of (i) yields noise and regularization calibrations that improve over the universal worst-case implied by stage-wise composition. We further bound the suboptimality gap of the private estimator relative to the non-private optimum (Theorem~\ref{thm:utility}), and show that under favorable weight-stability regimes ($\overline{W}_1,\overline{W}_2=O(\sqrt{n})$) DP-2ERM attains standard parametric rates---$O_\P(n^{-1/2})$ in parameter and $O_\P(n^{-1})$ in excess objective under strong convexity---so that privacy is obtained without degrading the asymptotic statistical rate (Theorem~\ref{thm:consistency}).
	\end{enumerate}
	
	We apply DP-2ERM to ITR learning and evaluate it in simulations and benchmark applications. In the simulations and benchmark application we report, DP-2ERM improves the privacy--utility trade-off relative to worst-case private baselines motivated by composition-style calibration, and more stable optimization-based balancing weights, particularly EBW, often lead to better private utility than standard IPW.  In addition, we show in Appendix  that the same weight-stability analysis yields differentially private estimation of ATE via a plug-in specialization.

	\section{Background and Related Work}
	
	\subsection{Differential Privacy}
	
	The mathematical framework of differential privacy was first introduced by Dwork, McSherry, Nissim, and Smith in \cite{dwork2006calibrating} to provide a rigorous foundation for privacy-preserving data analysis. Here, we provide a brief exposition of its core definitions. 
	
	Consider a randomized mechanism $\mathcal{M}$ that produces an output $\mathcal{M}(\mathcal{D})$ from a given dataset $\mathcal{D}$. The core principle of differential privacy is to ensure that the output distribution does not depend sensitively on the presence of any single individual, thereby protecting sensitive information against adversarial inference. This requirement can be viewed as distributional robustness under single data perturbations: the probability distributions of $\mathcal{M}(\mathcal{D})$ and $\mathcal{M}(\mathcal{D}')$ must remain statistically similar whenever $\mathcal{D}$ and $\mathcal{D}'$ are datasets of the same size $n$ that differ in exactly one entry. The following definition of $(\varepsilon, \delta)$-differential privacy formalizes this intuition \cite{dwork2009differential}.
	
	\begin{definition}[Differential Privacy]\label{def:DP}
		A randomized mechanism $\mathcal{M}$ satisfies \textit{$(\varepsilon, \delta)$-differential privacy} if, for any two datasets  $\mathcal{D}$ and $\mathcal{D}'$ of the same size $n$ that differ in exactly one entry, and for any measurable subset $S$ of the output space of $\mathcal{M}$, it holds that
		\begin{align}
			\P(\mathcal{M}(\mathcal{D}) \in S) \leq e^{\varepsilon} \P(\mathcal{M}(\mathcal{D}') \in S) + \delta.
		\end{align}
	\end{definition}
	
	The parameter $\varepsilon$ represents the privacy budget, quantifying the maximum allowable divergence between the output distributions of neighboring datasets; smaller values of $\varepsilon$ indicate stronger privacy. The parameter $\delta$ acts as a failure probability, relaxing the strict requirement of $\varepsilon$-DP by introducing an additive error term.
	
	\subsection{Differentially Private Empirical Risk Minimization}
	
	Chaudhuri, Monteleoni, and Sarwate \cite{chaudhuri2011differentially} introduced \emph{objective perturbation} for differentially private ERM in the standard (unweighted) setting, combining $L_{2}$-regularization with a random linear term added to the objective to guarantee differential privacy of the minimizer of the perturbed objective.
	
	However, the original objective perturbation analysis imposes restrictive regularity conditions, including convexity, differentiability, and strong convexity (typically enforced via $L_{2}$-regularization). As a result, it does not directly accommodate common nonsmooth formulations such as $L_{1}$-regularization for sparsity or other constrained/nonsmooth objectives without additional technical modifications. To address such settings, Kifer et al. \cite{kifer2012private} developed extensions that handle constrained optimization while maintaining privacy guarantees, and they also demonstrated how to leverage $(\varepsilon,\delta)$-DP with Gaussian perturbations. 
	
	A key limitation for our purposes is that these seminal works focus on equal-weight (unweighted) ERM. In contrast, 2ERM involves \emph{data-dependent} weights in the second-stage objective, creating a tighter coupling between the data, the weight map, and the final estimator. This coupling is common in causal inference, where data-dependent reweighting is routine—for example, through covariate balancing weights used to learn individualized treatment rules.

	\subsection{Covariate Balancing Weights and ITR}
	
	A motivating instance of 2ERM is learning ITRs from observational data. In observational studies, treatment assignment is confounded, so the treated and control groups typically differ in their covariate distributions; without adjustment, learning an ITR by na\"ive (unweighted) risk minimization can be biased toward regions of the covariate space that are over-represented in one arm.

	A standard approach is IPW, which weights each sample by the inverse of its propensity score. This requires specifying (and estimating) a propensity score model, and the resulting weights can be sensitive to model misspecification and limited overlap. For this reason, methods that directly estimate weights based on balance-seeking objectives have been proposed. These approaches include EBW \citep{hainmueller2012entropy}, stable balancing weights \citep{zubizarreta2015stable} as well as distributional covariate balancing weights \citep{huling2020energy, kallus2020generalized}. Distributional covariate balancing weights are modern alternatives to IPW that directly minimize a discrepancy between the empirical covariate distributions induced by the weights, thereby targeting covariate balance without committing to a parametric propensity model.
	
	All of the aforementioned covariate balancing weights 
	are computed by solving convex optimization problems of the form described in the first stage in \eqref{eq:find_optimal_weight}. These weights, once computed, are then used to formulate a weighted ERM problem in the second stage \eqref{eq:weighted_convex_opt} to estimate an ITR within a chosen decision-function class.
	
	This two-stage structure is precisely what motivates our 2ERM framework and our privacy analysis: a perturbation of one record can change the entire weight vector computed in the first stage, which in turn affects the sensitivity of the second-stage solution. We make the reduction to 2ERM explicit in Section~\ref{sec:running_example_ITR} and return to the private ITR instantiation in Section~\ref{sec:ITR}.

	\subsection{Related Work on Differentially Private Causal Inference and ITR Learning}
	\label{sec:related_work_DP_CI}
	
	We briefly review related work on differentially private causal inference and policy learning, none of which directly addresses the regime targeted by DP-2ERM in which optimization-based balancing weights are data-dependent solutions to convex programs, are not released, and feed into a second-stage weighted ERM under observational confounding.
	
	Niu et al. \cite{niu2022differentially} proposed a differentially private procedure for heterogeneous effect estimation that splits the data into three parts, privately estimates balancing weights and an outcome regression using budget $\varepsilon/3$ each, and then estimates the \textit{conditional average treatment effect} (CATE) using the remaining $\varepsilon/3$. Their approach is tailored to parametric weighting schemes, such as IPW based on a propensity model, and does not directly accommodate optimization-based balancing weights such as EBW and distributional balancing weights, which are typically obtained by solving convex programs. In addition, privatizing the first stage can degrade covariate balance, and allocating privacy budget across intermediate releases may lead to conservative noise levels and reduced utility.
	
	Giddens et al.\ \cite{giddens2023differentially} study weighted ERM under output perturbation and include OWL as an application. While conceptually related to our framework, their analysis relies on an asymptotic stability condition under which the sensitivity of the weighted ERM problem reduces to that of a standard unweighted ERM. Our setting is different: the main challenge is precisely to quantify how perturbations propagate through data-dependent balancing weights obtained from a first-stage optimization problem. In addition, their empirical focus is primarily on randomized clinical trials, where propensity scores are typically fixed and known.
	
	Spicker, Moodie, and Shortreed \cite{spicker2024differentially} study DP in the context of OWL for dynamic treatment regimes. Their approach privatizes a single-stage weighted SVM classifier, treating the propensity score weights as fixed and known inputs. As a result, their analysis does not confront the two-stage sensitivity amplification that is central to DP-2ERM.
	
	The closest concurrent work is Ohnishi and Awan \cite{ohnishi2024differentially}, who also consider a two-stage procedure involving covariate balancing weights. However, their target estimand is the ATE, a scalar quantity, rather than an ITR, and their framework privatizes and releases the weights themselves as an intermediate output. In contrast, DP-2ERM treats the balancing weights as internal nuisance parameters that are not released and establishes privacy directly for the final released estimator. We return in Section~\ref{sec:related_work_composition} to the specific limitations of stage-wise composition in the 2ERM setting.
	
	\subsection{Limitations of Stage-wise Composition for 2ERM}
	\label{sec:related_work_composition}
	
	A natural DP baseline for 2ERM is to privatize each stage separately and combine the guarantees by composition \citep{chaudhuri2011differentially, dwork2010boosting}. Concretely, one computes a private weight vector in the first stage with parameters $(\varepsilon_{1},\delta_{1})$ (e.g., via objective perturbation or noisy SGD), then solves the second-stage weighted ERM with an independent mechanism satisfying $(\varepsilon_{2},\delta_{2})$-DP. Standard composition yields an overall guarantee of at least $(\varepsilon_{1}+\varepsilon_{2},\delta_{1}+\delta_{2})$ \citep{dwork2006calibrating, dwork2009differential}. Sharper accounting methods (e.g., \citealp{dwork2010boosting, kairouz2015composition, dong2022gaussian}) improve constants and scaling under repeated compositions, but they still effectively require the second stage to be private \emph{uniformly} over all possible first-stage outputs. In 2ERM, those first-stage outputs are precisely the data-dependent weights entering the second-stage objective, so the required uniform guarantee is over a potentially large and highly heterogeneous class of weight vectors.
	
	A literal stage-wise composition strategy would also require privatizing the first-stage weight computation itself, even though these weights are internal nuisance quantities and are not released. In settings such as causal inference and ITR learning, perturbing the weights can directly undermine the covariate balancing constraints they are meant to enforce, which motivates the direct formulation we pursue in later sections, where privacy noise is injected only into the released second-stage estimator while the balancing role of the first stage is preserved.
	
	Realizing such a literal stage-wise composition baseline also requires a DP calibration for the second stage: weighted ERM with data-independent (fixed) weights. Existing DP-ERM results \citep{chaudhuri2011differentially,kifer2012private} cover equal-weight ERM under rank-$\le 2$ per-sample Hessians and do not directly provide such a calibration. The fixed-weight DP-ERM calibration used for the composition baseline in Section~\ref{sec:simulation} is therefore obtained as a corollary of our sensitivity analysis---which extends objective perturbation to weighted losses with arbitrary-rank Hessians controlled by the trace bound $\lambda$---combined with standard DP logistic regression \citep{chaudhuri2011differentially} for the first-stage fit.
	
	\section{Statement of Results}
	
	\subsection{Notations}\label{subsection:notation}
	
	In this paper, we use the notation $\R^{p}$ to represent the ambient space for data, equipped with the standard inner product $\langle\cdot, \cdot \rangle$, inducing the Euclidean norm $\lVert \cdot \rVert$. For each integer $n\ge 1$, $\I_{n}$ denotes the $n\times n$ identity matrix. We denote its Frobenius, operator (2-), and supremum norm by $\lVert \A \rVert_{F}^{2} := \sum_{i,j} a_{ij}^{2},  \lVert \A \rVert_{2} := \sup_{\x\in \R^{n},\, \lVert \x \rVert=1} \, \lVert \A\x \rVert,  \lVert \A \rVert_{\infty}:= \max_{i,j} |a_{ij}|$, respectively. For square symmetric matrices $\A,\B\in \R^{n\times n}$, $\A \preceq \B$ indicates that $\v^{\top}\A\v \le \v^{\top}\B\v $ holds for all unit vectors $\v\in \R^{n}$. For two scalars $a,b$, we denote $a\lor b = \max \{ a,b\}$ and $a\land b=\min\{a,b \}$, and we define the positive part of a scalar $u$ as $(u)_+ := u \lor 0$. For two datasets $\mathcal{D}$ and $\mathcal{D}'$, we write $\mathcal{D} \sim \mathcal{D}'$ to indicate that they are \textit{neighbors}, i.e., they differ in exactly one entry.
	
	\subsection{Problem Formulation: Two-Stage ERM}\label{sec:problem_formulation}
	Before presenting our private algorithm, we formally define the general 2ERM framework. This paradigm involves two-stage optimization problems defined as follows:
	
	\begin{enumerate}[label=(\roman*)]
		\item (\textbf{Find optimal sample weights}) Let $\Delta^{n}$ denote an $n$-dimensional probability simplex. Fix a closed convex parameter space $\Xi$ and a function $F:\Xi\rightarrow n\Delta^{n}$ that maps a parameter $\blambda \in \Xi$ to a nonnegative weight vector $F(\blambda)$ of total sum $n$.
		
		Find an optimal weight vector $\w(\mathcal{D})\in n\Delta^{n}$ for the given data $\mathcal{D}:=(\d_{1},\dots,\d_{n})$ by solving the following convex optimization problem
		\begin{align}\label{eq:find_optimal_weight}
			\w=\w(\mathcal{D}):=F(\blambda^{*}(\mathcal{D})), \quad \textup{where}\quad \blambda^{*}(\mathcal{D}) = \argmin_{\blambda\in \Xi} \,\,\mathcal{E}(\blambda; \mathcal{D}),
		\end{align}
		where $\mathcal{E}(\,\cdot\,;\mathcal{D})$ is a convex loss function depending on the observed data $\mathcal{D}$. (We allow trivial parameterization with $F=\textup{identity}$, e.g., MMD, see Section \ref{sec:MMD}.)
		
		\item (\textbf{Find optimal parameter}) Write $\w(\mathcal{D})=(w_{1},\dots,w_{n})$ in \eqref{eq:find_optimal_weight}. 
		Find the optimal parameter $\param$ by solving the following weighted empirical risk minimization
		\begin{align}\label{eq:weighted_convex_opt}
			\param^{\textup{opt}}(\mathcal{D}) = \argmin_{\param\in \Param} \, \left[ \mathcal{L}(\param; \mathcal{D}):= \frac{1}{n}\sum_{i=1}^{n} w_{i}  \cdot \ell(\param; \d_{i})  + R(\param) \right],
		\end{align}
		where $\ell(\,\cdot\,; \d_{i})$ is a convex per-sample loss function, $\Param\subseteq \R^{p}$ is a closed convex constraint set, and $R(\cdot)$ is a convex regularizer. 
	\end{enumerate}
	
	\subsection{Proposed Method: The DP-2ERM Algorithm}
	
	Recall that the standard ERM corresponds to the case where $\w(\mathcal{D})\equiv \mathbf{1}$. In 2ERM, the observed data $\mathcal{D}$ is used sequentially. To emphasize this dependence, we denote the optimal parameter as $\param^{\textup{opt}}(\mathcal{D})  = \param^{\textup{opt}}(\mathcal{D}; \w(\mathcal{D}))$.
	
	To solve the 2ERM problem under differential privacy constraints, we utilize objective perturbation introduced in \cite{chaudhuri2011differentially}, which is to add an $L_2$-regularization along with a random linear term to $\mathcal{L}(\theta;\mathcal{D})$:
	\begin{align}\label{eq:weighted_convex_opt_DP}
		\param^{\textup{priv}}&= \param^{\textup{priv}}(\mathcal{D},\b) = \argmin_{\param\in \Param} \, \left[ \mathcal{L}^{\textup{priv}}(\param; \mathcal{D},\b):= 
		\mathcal{L}(\param;\mathcal{D}) + \frac{\gamma}{2} \| \param \|_{2}^{2} + \frac{\langle \b, \param \rangle}{n} \right].
	\end{align}
	Here, $\b$ is a random noise vector sampled from a distribution (e.g., Gamma \cite{chaudhuri2011differentially} or Gaussian \cite{kifer2012private}) with scale explicitly calibrated to override the weight stability derived in our analysis, and $\gamma$ is a regularization parameter. The complete procedure is detailed in Algorithm \ref{algorithm:main}.
	
	\begin{algorithm}[H]
		\caption{DP-2ERM}
		\label{algorithm:main}
		\begin{algorithmic}[1]
			\REQUIRE Dataset $\mathcal{D}:=(\d_{1},\dots,\d_{n})$; privacy parameters \( \varepsilon \) and \( \delta \) (\( \delta = 0 \) for \(\varepsilon\)-differential privacy); closed and convex domain  $\Param \subseteq \mathbb{R}^{p}$; convex regularizer \( R(\param) \); convex loss function $\w\mapsto \mathcal{H}(\w;\mathcal{D})$ for weights; convex loss function $\param\mapsto \ell(\param;\d)$ for parameters with continuous Hessian; bound \( \zeta \) such that \( \|\nabla_{\param} \ell(\param; \d)\|_2 \leq \zeta \) for all \( \d \) and \( \param \in \Param\), and upper bound \( \lambda \) on the trace of Hessian \( \nabla^2_{\param} \ell(\param; \d) \) for all \( \d \) and \( \param \in \Param\).
			\ENSURE Private parameter estimate \( \param^{\text{priv}}  \in \Param \)
			\STATE Non-private optimal weights: $\w(\mathcal{D})=(w_{1},\dots,w_{n})\leftarrow \argmin_{\w\in n \Delta^{n}} \, \mathcal{H}\left(\w; \mathcal{D}\right)$ 
			\IF{\(\varepsilon\)-differential privacy is required}
			\STATE Sample \( \b \in \mathbb{R}^p \) from  Gamma distribution with density: $ \nu_{1}(\b; \beta) \propto e^{- \beta \|\b\|_2}$
			\ELSIF{\((\varepsilon, \delta)\)-differential privacy is required}
			\STATE Sample \( \b \in \mathbb{R}^p \) from  Gaussian distribution: $\displaystyle  \nu_{2}(\b; \sigma^{2}) = \mathcal{N}\left(\mathbf{0}, \sigma^2 \I\right)$ ($\sigma^2$ calibrated from the weight-stability bound; see Theorem~\ref{thm:stability_weights_gen} below)
			\ENDIF
			\STATE Private optimal parameter: $\param^{\text{priv}}\leftarrow \arg\min_{\param \in \Param} \frac{1}{n}\sum_{i=1}^{n} w_{i} \, \ell(\param; \d_{i})  + R(\param)  + \frac{\gamma}{2 } \|\param\|_2^2 + \frac{\langle \b, \param \rangle}{n} $  
			\RETURN \( \param^{\text{priv}} \)
		\end{algorithmic}
	\end{algorithm}
	
	To implement Algorithm~\ref{algorithm:main}, the practitioner must specify three ingredients: (i) assumptions on the covariate domain and loss function, from which $\zeta$ and $\lambda$ are obtained; (ii) a choice of weighting method together with an associated stability bound, from which $\overline{W}_{1}$ and $\overline{W}_{2}$ are obtained via Theorems~\ref{thm:stability_weights_gen} and~\ref{thm:DP_guarnatee}; and (iii) target privacy parameters $(\varepsilon,\delta)$, which in turn determine the calibrated noise level and regularization parameter $\gamma$ through Theorem~\ref{thm:DP_guarnatee}.

	\subsection{General Results on DP-2ERM}
	
	The core theoretical contribution of our analysis is the following probabilistic data sensitivity bound stated in Theorem \ref{thm:sensitivity_objective_perturbation} that characterizes the distributional perturbation of the private estimator $\param^{\textup{priv}}$ due to a change in a single data point.

	A key challenge arises from the algorithm's two-stage structure: changing a single data point ($\mathcal{D}\mapsto \mathcal{D}$') alters the private estimator $\param^{\textup{priv}}$  through two distinct channels—first by changing the weights in the first stage ($\w(\mathcal{D})\mapsto \w(\mathcal{D}')$), and second by directly affecting the training data in the weighted ERM in the second stage ($\mathcal{D}\mapsto \mathcal{D}'$). Consequently, the overall statistical property of the private estimator $\param^{\textup{priv}}(\mathcal{D}')$ critically depends on how much the weights are perturbed and also how sensitive the solution of the weighted ERM is to the input data. Our general data sensitivity bound in Theorem \ref{thm:sensitivity_objective_perturbation} formalizes this dependency, establishing a precise relationship between data perturbation and the resulting final private estimator $\param^{\textup{priv}}$. Proofs of results stated in this subsection are provided in Appendix,
	unless stated otherwise.

	\begin{theorem}[Data Sensitivity Bound for DP-2ERM via Objective Perturbation]\label{thm:sensitivity_objective_perturbation}
		Fix size-$n$ datasets $\mathcal{D}$ and $\mathcal{D}'$ that differ by at most one point and let $\w=(w_{1},\dots,w_{n})$ and $\w'=(w_{1}',\dots,w_{n}')$ be the non-private optimal sample weights solving \eqref{eq:find_optimal_weight} for the datasets $\mathcal{D}$ and $\mathcal{D}'$, respectively. Define $W_{1}=W_{1}(\mathcal{D},\mathcal{D'})$ and $W_{2}=W_{2}(\mathcal{D},\mathcal{D'})$ by 
		\begin{align}\label{eq:def_W1_W2}
			W_{1}&:=  \| \w-\w' \|_{1} + 2\max_{1\le i \le n}(w_{i}\land w_{i}'), \\ \nonumber
			W_{2}&:= \sqrt{\left(\|\w-\w' \|^{2}_{2}+ 2\max_{1\le i \le n} w_{i}w_{i}'\right)\left(1+\|\w-\w'\|_{0} \right)}.
		\end{align}
		Suppose 
		there exist $\lambda,\zeta$ such that $\textup{tr}(\nabla^{2} \ell(\param;\d))\le \lambda$ and  $ \|\nabla \ell(\param; \d)\|_2 \leq \zeta$ for all $\param\in \Param$ and $\d\in \mathcal{D}\cup \mathcal{D}'$.  
		Then for Algorithm \ref{algorithm:main} and for each Borel measurable set $\Param_{B}\subseteq \R^{p}$, we have  
		\begin{enumerate}[label=(\roman*)]
			\item With \textbf{Gamma-radial noise}
			$f_{\b}(\x)\propto \exp(-\beta \| \x \|_{2})$ (often informally called \textbf{Gamma noise} since $\|\b\|_{2}$ is Gamma-distributed under this law), for each Borel measurable set $\Param_{B}\subseteq \R^{p}$, 
			\begin{align}\label{eq:thmA1_gamma}
				\qquad   \P_{\mathcal{D}}(\param^{\textup{priv}}\in \Param_{B}) &\le \exp\left(  \beta  \zeta W_{1} +  \frac{\lambda}{\gamma n} W_{2}   \right) \P_{\mathcal{D}'}(\param^{\textup{priv}}\in \Param_{B}). 
			\end{align}
			\item With \textbf{Gaussian noise} $\b\sim \mathcal{N}(\mathbf{0},\sigma^{2}\I)$, for each $\delta>0$ and Borel measurable set $\Param_{B}\subseteq \R^{p}$, 
			\begin{align}\label{eq:thmA1_gaussian}
				\qquad \P_{\mathcal{D}}(\param^{\textup{priv}}\in \Param_{B})  &\le  \exp\left( \bar{\rho} W_{1}+ \frac{\lambda}{\gamma n} W_{2}   \right)  \P_{\mathcal{D}'}(\param^{\textup{priv}}\in \Param_{B}) + \delta, \nonumber \\
				\qquad \text{where } \bar{\rho} &:= \frac{1}{2\sigma^{2}} \left(2\zeta \sigma  \sqrt{ \left(\sqrt{p}+\sqrt{\log \delta^{-1}}\right)^{2} + \log \delta^{-1}}  + \zeta^{2}  \right).
			\end{align}

		\end{enumerate}   
		
	\end{theorem}
	
	\noindent\textit{Proof of Theorem~\ref{thm:sensitivity_objective_perturbation}.}
	See Appendix.
	
	From the above result, one can immediately deduce a sufficient noise level and $L_{2}$-regularization in the objective perturbation to ensure differential privacy of the entire DP-2ERM procedure. Indeed, for the Gamma noise, we can easily guarantee $\varepsilon$-DP of Algorithm \ref{algorithm:main} by choosing both the noise level $\beta^{-1}$
	and the $L_{2}$-regularization parameter $\gamma$ large enough so that the exponent in \eqref{eq:thmA1_gamma} is at most $\varepsilon$ for all possible neighboring datasets $\mathcal{D}$ and $\mathcal{D}'$. This gives us the following result on differential privacy guarantee of DP-2ERM.

	\begin{theorem}[Differential Privacy of DP-2ERM]\label{thm:DP_guarnatee}
		Suppose there exist $\lambda,\zeta$ such that $\textup{tr}(\nabla^{2} \ell(\param;\d))\le \lambda$ and  $ \|\nabla \ell(\param; \d)\|_2 \leq \zeta$ for all $\param\in \Param$ and for all feasible feature vectors $\d$. With $W_{i}$'s defined in \eqref{eq:def_W1_W2}, denote  
		\begin{align}
			\overline{W}_{i}:=\sup_{\mathcal{D}\sim \mathcal{D}'} W_{i}(\mathcal{D},\mathcal{D}')  \qquad \textup{for $i=1,2$},
		\end{align}
		where the supremums are over all size-$n$ datasets $\mathcal{D},\mathcal{D}'$ that differ by at most one point. Then the following holds for Algorithm \ref{algorithm:main}:
		\begin{enumerate}[label=(\roman*)]
			\item For the Gamma noise, 
			$\varepsilon$-DP is guaranteed provided 
			\begin{align}\label{eq:beta_gamma_choice}
				\frac{1}{\beta} \ge  \frac{2\zeta  }{ \varepsilon } \overline{W}_{1}\quad \textup{and}\quad \gamma \ge \frac{2\lambda}{\varepsilon n} \overline{W}_{2},
			\end{align}
			
			\item For the Gaussian noise, $(\varepsilon,\delta)$-DP is guaranteed provided
			\begin{align}\label{eq:sigma_Gaussian_choice}
				\sigma \ge \frac{\zeta}{\varepsilon} \left( \tilde{L} + \sqrt{\tilde{L}^{2} + \frac{\varepsilon}{\overline{W}_{1}}} \right) \overline{W}_{1} \quad \textup{and} \quad \gamma \ge \frac{2\lambda}{\varepsilon n} \overline{W}_{2},
			\end{align}
			where $\tilde{L}:=\sqrt{ (\sqrt{p}+\sqrt{\log \delta^{-1}})^{2} + \log \delta^{-1}}$. 
		\end{enumerate}
	\end{theorem}

	A key feature in our privacy guarantee in Theorem  \ref{thm:DP_guarnatee} above is that the required noise level and the regularization for the objective perturbation depend explicitly on the data sensitivity of the weights through the quantities $W_{1}$ and $W_{2}$. Intuitively, these quantities are larger for more data-sensitive weighting schemes, which in turn requires a larger noise level and regularization to guarantee privacy of DP-2ERM under the same privacy budget. 
	
	To make the privacy guarantee in Theorem \ref{thm:DP_guarnatee} practical, we need to bound the worst-case weight sensitivity measures $\overline{W}_{1}$ and $\overline{W}_{2}$. These quantities depend on the specific choice of the first-stage weighting mechanism (see Section \ref{sec:stability_CBW}). We first consider the simplest case: a single-stage weighted ERM where the weights $\mathbf{w}$ are constant, data-independent (e.g., $\mathbf{w} = (1, \dots, 1)$). In this scenario, Theorem \ref{thm:DP_guarnatee} simplifies as follows.
	
	\begin{corollary}[Differential Privacy of Data-Independent Weighted ERM]\label{cor:DP_guarnatee_weighted_ERM}
		Suppose data-independent weight $\w=(w_{1},\dots,w_{n})$. Then Theorem \ref{thm:DP_guarnatee} holds with $\overline{W}_{1}=2\max w_{i}$ and $\overline{W}_{2}=\sqrt{2} \max w_{i} $. 
	\end{corollary}
	
	Note that when $\w=(1,\dots,1)$, Corollary \ref{cor:DP_guarnatee_weighted_ERM} essentially recovers the DP-ERM privacy guarantee of Chaudhari et al. in \cite{chaudhuri2011differentially}. 
	
	Next, we derive a corollary of Theorem \ref{thm:DP_guarnatee} that depends on the first-stage weighting mechanism only through a uniform bound on its maximum weight. Note that the sum-to-$n$ constraint implies the universal weight-stability bound $\|\w-\w'\|_{2}\le \| \w-\w' \|_{1} \le 2n$ and the trivial uniform bound $\max_{i} w_{i}\le n$; any additional structural bound $\max_{i} w_{i}(\mathcal{D})\le R$ may be imposed on top (we write the hypothesis as $R\wedge n$ so that $R=n$ is always admissible and recovers the unconditional case). This yields the following DP guarantee for a general non-private weighting mechanism.
	
	\begin{corollary}[Universal Differential Privacy of 2ERM via Objective Perturbation]\label{cor:DP_objective_perturbation}
		Suppose the first-stage weighting mechanism satisfies $\max_{1\le i\le n} w_{i}(\mathcal{D}) \le R\land n$ (the sum-to-$n$ constraint guarantees $R \le n$). Then Theorem~\ref{thm:DP_guarnatee} holds with
		\[
		\overline{W}_{1} \;\le\; 2n + 2(R\wedge n), \qquad
		\overline{W}_{2} \;\le\; \sqrt{2n(R\wedge n)(1+n)},
		\]
		irrespective of the particular first-stage mechanism. In particular, it always holds that $\overline{W}_{1}\le 4n$ and $\overline{W}_{2}\le \sqrt{2}\,n\sqrt{1+n}$. 
\end{corollary}

While the algorithm setting in Corollary \ref{cor:DP_objective_perturbation} provides a universal differential privacy guarantee for DP-2ERM that is robust against the first-stage solution, the amount of noise and regularization provided there cannot guarantee strong utility. This is because these amounts were set high enough to prevent privacy leakage even in the most extreme weight shift scenario, such as achieving $\| \w-\w' \|_{1}=2n$ (e.g.,
$\w$ is supported on the first half of the coordinates and $\w'$ is supported on the second half of the coordinates). Even specializing Corollary~\ref{cor:DP_objective_perturbation} to a regime with a deterministic uniform bound $\max_{i} w_{i}(\mathcal{D})\le R\ll n$---which sharpens $\overline{W}_{2}$ from $O(n^{3/2})$ to $O(n\sqrt{R})$---still pays the full $\|\w-\w'\|_{1}\le 2n$ in the $\overline{W}_{1}$ term and therefore remains worst-case in the weight perturbation itself. Therefore, to apply Theorem \ref{thm:sensitivity_objective_perturbation} more efficiently, we need a tighter bound on the weight perturbation $\|\w-\w'\|$ itself, which must inevitably depend on the nature of the first-stage weighting mechanisms. We obtain more specific weight perturbation bounds for various standard weighting mechanisms in the causal inference literature, such as the IPW, MMD, and EBW. These results are stated in Section \ref{sec:stability_CBW} in detail.

We now turn our attention to the utility guarantees.
As long as $(\varepsilon, \delta)$-DP is guaranteed for the entire procedure, it would be beneficial to choose both the noise level ($1/\beta$ for the Gamma noise and $\sigma$ for the Gaussian noise) and the $L_{2}$-regularization parameter $\gamma$ in \eqref{eq:beta_gamma_choice} as small as possible to minimize the effect of unnecessary noise and estimation error due to severe $L_{2}$-regularization. This can be justified through utility analysis. Namely, compared to the smallest possible objective value of $\mathcal{L}(\cdot)$ at the optimal non-private parameter $\param^{\textup{opt}}$, how small is the objective value at the private estimate $\param^{\textup{priv}}$? In the following result, we show that the suboptimality of the private estimate $\param^{\textup{priv}}$ is bounded by a decreasing function of $\beta$ and $\gamma^{-1}$ and an increasing function of $\sigma$.

\begin{theorem}[Utility Guarantees for DP-2ERM]\label{thm:utility}
	Suppose $\mathcal{D}$ is a size-$n$ dataset and let $\w=(w_{1},\dots,w_{n})$ be the non-private optimal
	weights solving \eqref{eq:find_optimal_weight} for dataset $\mathcal{D}$. Assume that the objective $\mathcal{L}$ is $\rho$-strongly convex for some $\rho\ge 0$ (allowing $\rho=0$). Denote $a_{0}:=  \| \param^{\textup{opt}} \|_{2}$. Then for each
	$t \ge C$,
	\begin{align}\label{eq:utility_reduction}
		\P_{\b}\!\left(  \mathcal{L}(\param^{\textup{priv}};\mathcal{D})\ge
		\mathcal{L}(\param^{\textup{opt}};\mathcal{D}) + t  \right)
		\le
		\P_{\b}\!\left( \|\b \|_{2} \ge nD_t \right),
	\end{align}
	where $\b\in \R^{p}$ denotes the noise vector and
	$D_{t}:= \frac{2(t-C)}{B+\sqrt{B^2 + 4A(t-C)}}$, with constants defined as $A=\frac{2\rho+3\gamma}{2(\rho+\gamma)^2}$,  $B=\frac{ a_{0} \gamma (\rho+3\gamma)  }{(\rho+\gamma)^2 }$, and  $C= \frac{2a_{0}^2(\rho+2\gamma)}{(\rho+\gamma)^2} \gamma^{2}$.
	Moreover,
	\begin{enumerate}[label=(\roman*)]
		\item If the pdf of $\b$ is $f_{\b}(\x)\propto \exp(-\beta \| \x\|_{2})$, then 
		\begin{align}\label{eq:utility_gamma_poly}
			\P_{\b}\left(  \mathcal{L}(\param^{\textup{priv}};\mathcal{D})\ge
			\mathcal{L}(\param^{\textup{opt}};\mathcal{D}) + t  \right)
			\le C_p \exp\!\left(-\frac{\beta nD_t}{2}\right),
		\end{align}
		where $C_p := e\left(\frac{2(p-1)}{e}\right)^{p-1}$ depends only on $p$ (and $C_1:=1$). 
		
		\item If $\b\sim \mathcal{N}(\mathbf{0}, \sigma^{2}\I_p)$, then 
		\begin{align}\label{eq:utility_gaussian}
			\P_{\b}\left(  \mathcal{L}(\param^{\textup{priv}};\mathcal{D})\ge  \mathcal{L}(\param^{\textup{opt}};\mathcal{D}) + t   \right) 
			&\le 2\exp\left( -\frac{1}{2\sigma^{2}} (n D_{t}  - \sigma\sqrt{p})_{+}^{2}  \right).  
		\end{align}
	\end{enumerate}
\end{theorem}

We may further simplify the tail bounds in Theorem~\ref{thm:utility} by substituting the privacy calibrations from Theorem~\ref{thm:DP_guarnatee}. 

\begin{remark}[Privacy--Utility Trade-off]\label{rem:tradeoff}
	Theorem~\ref{thm:utility} bounds only the \emph{excess} objective value 
	$\mathcal{L}(\param^{\textup{priv}};\mathcal{D})-\mathcal{L}(\param^{\textup{opt}};\mathcal{D})$ 
	incurred by privatization; it does not address the absolute performance 
	$\mathcal{L}(\param^{\textup{opt}};\mathcal{D})$ of the underlying non-private estimator, 
	which, in the ITR setting, depends on the covariate balancing effectiveness of the chosen first-stage weighting method. This distinction reveals a fundamental trade-off. On one 
	extreme, data-independent weights (e.g., $w_i\equiv 1$) have zero data sensitivity, so the required noise depends only on the second-stage loss; however, without covariate balancing, the non-private solution $\param^{\textup{opt}}$ is itself biased, yielding a 
	low baseline value $\mathcal{L}(\param^{\textup{opt}};\mathcal{D})$. On the other extreme, 
	data-dependent balancing weights (e.g., IPW, EBW, MMD) effectively correct for confounding and
	improve $\mathcal{L}(\param^{\textup{opt}};\mathcal{D})$, but their nonzero data sensitivity 
	demands a larger noise injection, widening the gap 
	$\mathcal{L}(\param^{\textup{priv}};\mathcal{D})-\mathcal{L}(\param^{\textup{opt}};\mathcal{D})$. 
	The overall private utility is determined by the balance between these two competing effects.
\end{remark}

\begin{remark}[Privacy--Utility Scaling]\label{rem:scaling}
	Suppose $\overline{W}_1,\overline{W}_2=O(\sqrt{n})$ (as in Corollary~\ref{cor:DP_improved}) and calibrate $(\beta,\sigma,\gamma)$ according to Theorem~\ref{thm:DP_guarnatee}. Then the tail bounds in Theorem~\ref{thm:utility} scale differently for the Gamma and the Gaussian noise, with a faster decay for the Gaussian case but carrying the usual dimensional dependence through $p$. Details are given in Appendix. 
\end{remark}

Lastly in this section, we analyze how the ``excess-utility'' $\mathcal{L}(\param^{\textup{priv}}; \mathcal{D})-\mathcal{L}(\param^{\textup{opt}}; \mathcal{D})$ scales as the size $n$ of the dataset $\mathcal{D}$ tends to the infinity. We will assume all hyperparameters, such as $p,\lambda,\zeta$, are fixed while the size $n$ of the datasets grows to infinity. By doing so, we will obtain large-sample consistency of the DP-2ERM estimator under mild conditions. 

We first observe that the universal worst-case bounds in Corollary \ref{cor:DP_objective_perturbation} are too crude to yield vanishing excess-utility. 

\begin{remark}
	Suppose $\overline{W}_1\sim n$ and $\overline{W}_{2}\sim n^{3/2}$ as in the universal worst-case calibration in Corollary \ref{cor:DP_objective_perturbation} and calibrate $(\beta,\sigma,\gamma)$ according to Theorem~\ref{thm:DP_guarnatee}.  Then Theorem \ref{thm:utility} implies $\mathcal{L}(\param^{\textup{priv}}; \mathcal{D})-\mathcal{L}(\param^{\textup{opt}}; \mathcal{D})=O_{\P}(\sqrt{n})$. 
\end{remark}

Next, assuming that we have an improved bound on the quantities $\overline{W}_{1},\overline{W}_{2}$  on first-stage weight stability to $O(\sqrt{n})$,
we show that the excess-utility vanishes with high probability at rate at least $1/\sqrt{n}$. If the second-stage objective $\mathcal{L}$ is already $\rho$-strongly convex for some $\rho>0$, then this implies that our DP-2ERM estimator given by Theorem \ref{thm:DP_guarnatee}, while satisfying the required $\varepsilon$-DP and $(\varepsilon,\delta)$-DP depending on the noise type, is an asymptotically consistent estimator of the non-private optimal parameter $\param^{\textup{opt}}$ with estimation error at most order $1/\sqrt{n}$. 

\begin{theorem}[Consistency of the DP-2ERM  estimator]\label{thm:consistency}
	Keep the same settings in Theorems \ref{thm:DP_guarnatee} and \ref{thm:utility}. 
	Suppose $\overline{W}_1,\overline{W}_2=O(\sqrt{n})$ 
	and calibrate $(\beta,\sigma,\gamma)$ according to Theorem~\ref{thm:DP_guarnatee} so that $\param^{\textup{priv}}$ is a $\varepsilon$-DP (resp., $(\varepsilon,\delta)$-DP) estimator of $\param^{\textup{opt}}$ with the Gamma (resp., Gaussian) noise. Then the following holds for both Gamma and Gaussian noise distributions.
	
	\begin{enumerate}[label=(\roman*), leftmargin=*]
		\item \textit{Intrinsic strong convexity ($\rho>0$):}
		if the second-stage objective $\mathcal{L}$ is already $\rho$-strongly convex, then
		\begin{align}\label{eq:consistency_strong_convex}
			\mathcal{L}(\param^{\textup{priv}}; \mathcal{D})-\mathcal{L}(\param^{\textup{opt}}; \mathcal{D}) = O_{\P}(n^{-1})  \quad \textup{and}\quad  \|  \param^{\textup{priv}} - \param^{\textup{opt}} \|_{2} = O_{\P}(n^{-1/2}). 
		\end{align}
		
		\item  \textit{No intrinsic strong convexity ($\rho=0$):}
		if curvature is supplied only through the added regularization $\gamma$, then 
		\begin{align}\label{eq:consistency_convex}
			\mathcal{L}(\param^{\textup{priv}}; \mathcal{D})-\mathcal{L}(\param^{\textup{opt}}; \mathcal{D}) = O_{\P}(n^{-1/2}).  
		\end{align}
	\end{enumerate}
\end{theorem}

Thus, the practical message of Theorem \ref{thm:utility} is that the utility of DP-2ERM depends not only on the privacy
mechanism but also on the stability of the first-stage weighting method. When the chosen first-stage weighting method yields $\overline{W}_1,\overline{W}_2=O(\sqrt{n})$, the excess-utility vanishes with high probability, while it could potentially blow up at rate $O(\sqrt{n})$ under universal worst-case calibration. In Theorem \ref{thm:stability_weights_gen} and Corollary \ref{cor:DP_improved} stated in the following section, we will justify the assumption $\overline{W}_{1},\overline{W}_{2}=O(\sqrt{n})$ for a wide range of covariate balancing weights. 

Before developing weight stability bounds, we briefly show how learning ITRs from observational data naturally induces a two-stage weighted ERM structure, making $\|\w(\mathcal{D})-\w(\mathcal{D}')\|$ the central quantity governing the privacy--utility trade-off.

\subsection{Running Example: ITR Learning as a 2ERM Problem}\label{sec:running_example_ITR}

We sketch how ITR estimation fits into the 2ERM framework and why covariate balancing weights are needed (see Section \ref{sec:ITR} for more details). Let $A\in\{0,1\}$ denote treatment, $Y$ the outcome, and $X$ baseline covariates. Let $Y(a)$ denote the potential outcome under treatment $a\in\mathcal{A}$. We assume the standard identification conditions in the potential outcome framework \cite{rubin1980randomization}: (i) the \textit{stable unit treatment value assumption} (SUTVA, $Y=Y(A)$), (ii) positivity ($0<\pi(a,\x):=\P(A=a\mid \X=\x)<1$ for all $a\in\mathcal{A}$ and $\x\in\mathcal{X}$), and (iii) no unmeasured confounding ($Y(a)\perp A\mid \X$ for all $a\in\mathcal{A}$).

Under the identification conditions, for a decision rule $d:\mathcal{X}\to\{0,1\}$, its value can be expressed as
\[
V(d)=\E\!\left[\frac{Y\,\mathbf{1}\{A=d(\X)\}}{\pi(A,\X)}\right]
\]
where $\pi(a,X)=\P(A=a\mid X)$ is the propensity score. In observational studies, the treated and control groups typically differ in their covariate distributions, and the inverse-propensity reweighting in $V(d)$ highlights that na\"ive (unweighted) learning can target the wrong covariate distribution.

In practice, $\pi$ is unknown and directly modeling it can be unstable. A modern alternative is to compute covariate balancing weights that directly rebalance treated and control covariate distributions. These weights are obtained in a first stage by solving a convex problem of the form \eqref{eq:find_optimal_weight}, yielding a data-dependent weight vector $\w(\mathcal{D})$. The ITR is then estimated in a second stage by solving a weighted ERM problem of the form \eqref{eq:weighted_convex_opt}, treating $\w(\mathcal{D})$ as a nuisance input. This is precisely the two-stage dependence captured by DP-2ERM: a single-record perturbation affects $\param^{\textup{priv}}$ both directly through the second-stage empirical risk and indirectly through the first-stage weights. Consequently, bounding $\|\w(\mathcal{D})-\w(\mathcal{D}')\|$ is the key step toward calibrating privacy noise with minimal loss of utility.

\section{Stability of Covariate Balancing Weights}
\label{sec:stability_CBW}

This section presents the paper's second main result: deterministic non-asymptotic stability bounds for the first-stage weight maps of three standard covariate-balancing schemes --- IPW, MMD, and EBW. Each bound quantifies how much $\w(\mathcal{D})$ can change when one record in $\mathcal{D}$ is modified, and is derived directly from the convex structure of the first-stage program without reference to privacy. Plugged into the sensitivity bound of Theorem~\ref{thm:sensitivity_objective_perturbation}, each yields a method-specific DP calibration (Corollary~\ref{cor:DP_improved}). We first state a general stability principle for weights defined as solutions to strongly convex optimization problems, then instantiate it for each method.

Our key tool for this weight stability analysis is the following result, which provides a general stability bound on how much the optimal weights $\w(\mathcal{D})$ are perturbed when a single point in the dataset $\mathcal{D}$ is altered. Specifically, it states that the perturbation is small if the change in the gradient of the objective $\mathcal{E}(\cdot;\mathcal{D})$ is small relative to its strong convexity parameter.

\begin{theorem}[General Stability of Parameterized Weights]\label{thm:stability_weights_gen1}
	Let  $\mathcal{D}$ and $\mathcal{D}'$ be two $n$-sample datasets. Let $\w(\mathcal{D})$ and $\w(\mathcal{D}')$ denote the corresponding optimal weights as in \eqref{eq:find_optimal_weight}. 
	Suppose $F$ is $L_{F}$-Lipschitz continuous and  $\mathcal{E}(\cdot;\mathcal{D})$ is $\rho$-strongly convex for some $\rho>0$. Then 
	\begin{align}\label{eq:weight_stability_bd_gen}
		\| \w(\mathcal{D}) - \w(\mathcal{D}')  \|_{2} \le \frac{2L_{F}}{\rho} \sup_{\blambda\in \Xi} \| \nabla \mathcal{E}(\blambda;\mathcal{D}) - \nabla \mathcal{E}(\blambda;\mathcal{D}') \|_{2}.  
	\end{align}
\end{theorem}

\noindent\textit{Proof of Theorem~\ref{thm:stability_weights_gen1}.}
See Appendix.

Based on the general principle in Theorem~\ref{thm:stability_weights_gen1}, we obtain stability bounds for various covariate balancing weights; IPW with randomized trials (Proposition~\ref{prop:stability_IPW_randomized_trial}), IPW with logistic regression (Propositions~\ref{prop:IPW_known_beta} and \ref{prop:IPW_estimated_stability}), MMD (Proposition~\ref{prop:MMD_stability}), and EBW (Theorem~\ref{thm:entropy_weight_stability}). Proofs are deferred to Appendix.
Below in Theorem~\ref{thm:stability_weights_gen}, we summarize these results with many constants suppressed in the $O(\cdot)$ notation. 

\begin{theorem}[Stability of Covariate Balancing Weights]\label{thm:stability_weights_gen}
	Let  $\mathcal{D}$ and $\mathcal{D}'$ be two $n$-sample datasets that differ at one point. Then we have 
	\begin{align}\label{eq:weight_stability_bd_improved}
		&\vspace{-1.5cm}\| \w(\mathcal{D})-\w(\mathcal{D}') \|_{2} \\
		&\hspace{-0.3cm}\le \begin{cases} 
			O(1)  & \textup{IPW for randomized trials; Proposition ~\ref{prop:stability_IPW_randomized_trial}}, \\
			O(1)  & \textup{IPW with known regression parameter; Prop. ~\ref{prop:IPW_known_beta}}, \\
			O\left(1\land \frac{1}{\lambda_{\min}(\hat{\Sigma})+\lambda_\textup{IPW}}\right)  & \textup{IPW with logistic regression; Proposition ~\ref{prop:IPW_estimated_stability}}, \\
			O(n/\lambda_{\textup{MMD}}) & \textup{MMD; Proposition ~\ref{prop:MMD_stability}}, \\
			O\left(1\land \frac{1}{(\lambda_{\min}(\hat{\Sigma})+\lambda_\textup{EBW})}\right) & \textup{EBW; Theorem ~\ref{thm:entropy_weight_stability}},
		\end{cases}
		\nonumber
	\end{align}
	where $\hat{\Sigma}\in \R^{p\times p}$ denotes the empirical covariance matrix of the covariates in $\mathcal{D}$. 
\end{theorem}

To avoid confusion with the Hessian trace bound $\lambda$ in Theorem \ref{thm:sensitivity_objective_perturbation}, we note that these $\lambda$ terms represent the $L_{2}$-regularization coefficients in the weight optimization problems defined in Equations \eqref{eq:IPW_beta_logistic_loss}, \eqref{eq:MMD_weight_def}, and \eqref{eq:entropy_dual}, respectively. For IPW with estimated logistic parameters and EBW, the bounds
also depend on the smallest eigenvalue of the empirical covariance matrix $\hat{\Sigma}$; when $\lambda_{\min}(\hat{\Sigma})$ is bounded away from zero (e.g., under i.i.d.\ sampling from a distribution with positive definite covariance), the rates simplify to 
$O(1\land (1+\lambda_{\textup{IPW}})^{-1})$ and $O(1\land (1+\lambda_{\textup{EBW}})^{-1})$, respectively. 
With positive regularization, these bounds are not directly comparable since the regularized 
parameters differ across methods.


By combining Theorems \ref{thm:DP_guarnatee} and \ref{thm:stability_weights_gen}, we can obtain significantly improved privacy guarantee of DP-2ERM with covariate balancing weights in the sense that, for the same privacy budget, the required noise level (proportional to $\overline{W}_{1}$) and the $L_{2}$ regularization (proportional to $\overline{W}_{2}$) in the objective perturbation for the second stage are much smaller. The improvement is at least a factor of $1/\sqrt{n}$ for the noise level and a factor of $1/n$ for the regularization. These results are stated in the corollary below, where the constants in $O(\cdot)$ terms can be traced from the corresponding results stated in \eqref{eq:weight_stability_bd_improved}. 

\begin{corollary}[Improved Privacy Guarantees of DP-2ERM with Covariate Balancing Weights]\label{cor:DP_improved}
	Theorem \ref{thm:DP_guarnatee} holds with the following improved bounds on $\overline{W}_{1}$ and $\overline{W}_{2}$:
	\begin{enumerate}[label=(\roman*)]
		\item If $\w(\mathcal{D})$ is the IPW with either randomized trial or known regression parameter, then $\overline{W}_{1}=O(1)$ and $\overline{W}_{2}=O(\sqrt{n})$. 
		
		\item If $\w(\mathcal{D})$ is the IPW with estimated regression parameter, then \\ $\overline{W}_{1}=O\!\left(\sqrt{n}\land\frac{\sqrt{n}}{\lambda_{\min}(\hat{\Sigma})+\lambda_{\textup{IPW}}}+1\right)$ and $\overline{W}_{2}=O\!\left(\sqrt{n}\land\frac{\sqrt{n}}{\lambda_{\min}(\hat{\Sigma})+\lambda_{\textup{IPW}}}+\sqrt{n}\right)$.
		
		\item If $\w(\mathcal{D})$ is the MMD, then
		$\overline{W}_{1}=O\!\left(\frac{n^{3/2}}{\lambda_{\textup{MMD}}}+1\right)$ and $\overline{W}_{2}=O\!\left(\frac{n^{3/2}}{\lambda_{\textup{MMD}}}+\sqrt{n}\right)$.
		
		\item If $\w(\mathcal{D})$ is the EBW, then 
		$\overline{W}_{1}=O\!\left(\sqrt{n}\land\frac{\sqrt{n}}{\lambda_{\min}(\hat{\Sigma})+\lambda_{\textup{EBW}}}+1\right)$ and \\$\overline{W}_{2}=O\!\left(\sqrt{n}\land\frac{\sqrt{n}}{\lambda_{\min}(\hat{\Sigma})+\lambda_{\textup{EBW}}}+\sqrt{n}\right)$.
	\end{enumerate}
\end{corollary}

Items (ii) and (iv) are structurally identical, reflecting the fact that both IPW with estimated logistic parameters and EBW achieve $O(1)$ stability when their respective regularization and covariance conditions hold. The key distinction is that the EBW bound arises from the tighter Lipschitz constant of the softmax normalization ($O(1/n)$ versus $O(1/\sqrt{n})$ for IPW), which makes the $O(1/\sqrt{n})$ parameter-shift term in the EBW stability bound vanish, whereas for IPW the analogous term remains $O(1)$ and is absorbed into the leading constant. In contrast, item (iii) shows that MMD requires $\lambda_{\textup{MMD}} = O(n)$ merely to achieve $O(\sqrt{n})$ noise scaling, reflecting the global coupling inherent in the kernel Gram matrix structure. In particular, if $\lambda_{\min}(\hat{\Sigma})$ is bounded away from zero, then the bounds on $\overline{W}_{1}$ in items (ii) and (iv) simplify to $O\left(\frac{\sqrt{n}}{1+\lambda_{\textup{IPW}}} + 1\right)$ and $O\left(\frac{\sqrt{n}}{1+\lambda_{\textup{EBW}}} + 1\right)$, respectively. For fixed or small regularization parameters, both scale as $O(\sqrt{n})$, but the constants implicit in the EBW bound are more favorable due to the tighter underlying stability.

\subsection{Stability of Weights in Randomized Trials}

First, we look at two instances of covariate balancing weights that can be found without any additional optimization procedure in the randomized trial setting. These are IPW under a randomized trial and IPW with a known regression parameter.

\subsubsection{Stability of IPW under Randomized Trial}

Consider an observed dataset $\mathcal{D}:=(\x_{i}, a_{i})_{i=1}^{n}$ from a randomized trial, where $\x_{i}$ denotes the covariate vector of patient $i$ and $a_{i}\in\{0,1\}$ denotes the treatment assignment for patient $i$. Define the \textit{propensity score} by
\begin{align}
	\pi(a\mid \x)=\P(A=a\mid X=\x), \qquad a\in\{0,1\}.
\end{align}
In a randomized trial, the treatment assignment is independent of the covariates, so the propensity score does not depend on $\x$. Hence, for each $a\in\{0,1\}$, $\pi(a\mid \x)=p_a, \text{for all } \x$, where $p_a=\P(A=a)$ is the marginal probability of receiving treatment $a$.

The IPW vector $\w=(w_1,\dots,w_n)$ is then defined as
\begin{align}
	\w
	=
	C\cdot
	\left(
	\frac{1}{p_{a_1}},\ldots,\frac{1}{p_{a_n}}
	\right),
\end{align}
where $p_{a_i}$ is the probability of receiving the observed treatment $a_i$, and $C$ is a normalization constant chosen so that $\sum_{i=1}^n w_i=n$, namely,
$
C=\frac{n}{\sum_{i=1}^n p_{a_i}^{-1}}.
$

Now consider a perturbed dataset $\mathcal{D}'$ of size $n$ that differs from $\mathcal{D}$ only in the last observation, with $(\x'_n, a'_n)$ replacing $(\x_n, a_n)$. The new IPW vector, say $\w'$, will differ from the previous one $\w$ if $a_{n}\ne a_{n}'$ and $p_{0}\ne p_{1}$. In this case, both the normalization constant $C$ and the last coordinate of the unnormalized weight vectors, $\frac{1}{p_{a_{n}}}$ and $\frac{1}{p_{a_{n}'}}$, change. We can bound their difference as follows.

\begin{prop}[Stability of IPW under Randomized Trial]\label{prop:stability_IPW_randomized_trial}
	
	Let $\w$ and $\w'$ denote the normalized IPW vectors under a randomized trial corresponding to datasets $\mathcal{D}$ and $\mathcal{D}'$ of size $n$ that differ at one point, as above. Then
	\begin{align}
		\| \w - \w' \|_{2} \le  \frac{\sqrt{2} \cdot |p_1 - p_0|}{p_0 \wedge p_1} \, \frac{n}{n-1}.
	\end{align}
	In particular, if the treatment assignment is unchanged in the differing record, then $\w = \w'$.
\end{prop}

This demonstrates that even in randomized trials, there remains a non-zero difference between the weight vectors when treatment assignments change. It highlights our subsequent development of a carefully designed privacy-preserving algorithm for ITR learning that addresses these sensitivity issues.

\subsubsection{Stability of IPW with Known Regression Parameters}

Next, consider the correctly specified case via logistic regression. Namely, we assume the true propensity score model is known and is given as
\begin{align}\label{eq:true_propensity_score}
	\mathbb{P}(A_i = 1 \mid X = \x_i) = \frac{\exp(\langle \x_{i},\blambda^* \rangle)}{1 + \exp(\langle \x_{i},\blambda^* \rangle)},
\end{align}
where $\blambda^* \in \mathbb{R}^p$ is the known true regression parameter vector. Under this propensity score model, the IPW for patient $i$ based on their observed treatment $a_i \in \{0,1\}$ is defined as
\begin{align}\label{eq:inverse_weight}
	\frac{1}{\mathbb{P}(A_i = a_i \mid X = \x_i)} = 1 + \exp(-(2a_i-1)\langle \x_i, \blambda^* \rangle).
\end{align}
Let $f_i = 1 + \exp(-(2a_i-1)\langle \x_i, \blambda^* \rangle)$. The weight vector $\w=(w_{1},\dots,w_{n})$ is obtained by normalizing these so their sum is $n$
\begin{align}
	w_{i} = n \frac{f_i}{\sum_{j=1}^{n} f_j}, \quad i=1,2,\dots,n.
\end{align}

\begin{prop}[Stability of IPW with Known Regression Parameters] \label{prop:IPW_known_beta}
	Let $\w$ and $\w'$ denote the normalized IPW vectors for datasets $\mathcal{D}$ and $\mathcal{D}'$ that differ at a single observation, as above, under the observational study with known $\blambda^{*}$. Suppose that the $L_{2}$-norm of each covariate vector is bounded by some constant $M$ (i.e., $\|\x_i\|_2 \le M$) and the true parameter is bounded by $R_\textup{IPW}$ (i.e., $\|\blambda^*\|_2 \le R_\textup{IPW}$). Then
	\begin{align}
		\|\w - \w' \|_2 \le \sqrt{2} \exp(MR_\textup{IPW}).
	\end{align}
\end{prop}

\subsection{Stability of Weights in Observational Studies}

The more interesting and more adaptive framework is when the covariate balancing weight vector $\w$ for a dataset $\mathcal{D}$ is parameterized by $\blambda$, and such a parameter is found by a separate convex optimization using the dataset $\mathcal{D}$. This covers the most interesting cases, including (1) IPW under observational study with estimated regression parameters, (2) MMD, and (3) EBW.

\subsubsection{Stability of IPW under Observational Study with Estimated Regression Parameters}

Suppose that we use logistic regression to estimate the propensity scores for all individuals. 
The estimated regression parameter $\hat{\blambda}$ is given by 
\begin{align}\label{eq:IPW_beta_logistic_loss}
	\hat{\blambda} = \argmin_{\| \blambda \|_{2} \le R_\textup{IPW}} \bigg[  \mathcal{E}^{i}(\blambda;\mathcal{D}):= \frac{1}{n} \sum_{i=1}^{n}  \underbrace{\log\left(1 + \exp(\langle \x_{i},\blambda \rangle \right)) - a_{i} \langle \x_{i},\blambda \rangle  }_{=:\ell_\textup{nll}(\blambda;\, \x_{i},a_{i})} + \frac{\lambda_{\textup{IPW}}}{2}  \| \blambda \|_{2}^{2}   \bigg],
\end{align}
where $\ell_\textup{nll}(\blambda;\, \x_{i},a_{i})$ above is the negative log-likelihood of observed data $(\x_{i},a_{i})$ under $\blambda$. To improve stability—a common practice when fitting IPW—we have added an $L_{2}$-regularization term with coefficient $\lambda_{\textup{IPW}}\ge 0$. Once we compute $\hat{\blambda}$, we then assign the normalized IPW weights for all $n$ patients as
\begin{align}\label{eq:IPW_weight_def}
	\w = (w_{1},\dots,w_{n}),\qquad \textup{where}\quad  w_{i}:= n \frac{
		1 + \exp(-(2a_i-1)\langle \x_i, \hat{\blambda} \rangle)}{\sum_{j=1}^{n} (
		1 + \exp(-(2a_j-1)\langle \x_j, \hat{\blambda} \rangle))}
\end{align}

\begin{prop}[Stability of IPW under Observational Study with Estimated Regression Parameters]\label{prop:IPW_estimated_stability}
	Let $\w$ and $\w'$ denote the IPW weights for datasets $\mathcal{D}$ and $\mathcal{D}'$ that differ at one point, as above, under the observational study with estimated $\hat{\blambda}$ solving \eqref{eq:IPW_beta_logistic_loss}. Suppose that the $L_{2}$-norm of each covariate vector is uniformly bounded by some constant $M$ (i.e., $\|\x_i\|_2 \le M$), the $L_{2}$-norm of the regression coefficients are bounded by $R_\textup{IPW}$ (i.e., $\|\blambda\|_2 \le R_\textup{IPW}$), and that the sample covariance matrix $\hat{\Sigma}:=n^{-1}\sum_{i=1}^{n} \x_{i}\x_{i}^{\top}$ satisfies $\lambda_{\min}(\hat{\Sigma})>0$. Define
	\begin{align}
		\rho := \frac{\exp(-MR_\textup{IPW})}{(1+\exp(-MR_\textup{IPW}))^{2}} \, \lambda_{\min}(\hat{\Sigma}) + \lambda_\textup{IPW}, \qquad \delta_{I} := \left( R_\textup{IPW} \wedge \, \frac{M/2}{\rho}\right).
	\end{align}
	Then,
	\begin{align}\label{eq:IPW_sensitivity_bd}
		\| \w(\mathcal{D}) - \w(\mathcal{D}')\|_{2} &\le \frac{4M^2 \exp(MR_\textup{IPW})}{\rho} + 2M \delta_{I} \exp(M \delta_{I}) \left(1+\frac{1+\exp(M\delta_{I})}{2n}\right).
	\end{align}
\end{prop}

\subsubsection{Stability of MMD Weights}\label{sec:MMD}

MMD is a particular type of distributional covariate balancing method (see \cite{gretton2012kernel, chen2024robust}), which are obtained as follows. 

Let $\mathcal{X}$ be a non-empty set and let $\mathcal{P}(\mathcal{X})$ denote the set of all probability measures defined on $\mathcal{X}$. Let $\mathcal{H}$ be a reproducing kernel Hilbert space with associated kernel $K: \mathcal{X} \times \mathcal{X} \rightarrow \mathbb{R}$. For $P, Q \in \mathcal{P}(\mathcal{X})$, the MMD between $P$ and $Q$ is defined as
\begin{align}
	\text{MMD}_{\mathcal{H}}(P, Q) := \sup_{f \in \mathcal{H}, \|f\|_{\mathcal{H}} \leq 1} |\mathbb{E}_{X \sim P}[f(X)] - \mathbb{E}_{Y \sim Q}[f(Y)]|.
\end{align}
It is well-known that (e.g., \cite[Lemma 6]{gretton2012kernel})
\begin{align}\label{eq:MMD_def_equiv}
	\text{MMD}_{\mathcal{H}}^2(P, Q) = \mathbb{E}_{X,X' \sim P}[K(X,X')]  + \mathbb{E}_{Y,Y' \sim Q}[K(Y,Y')] - 2\mathbb{E}_{X \sim P, Y \sim Q}[K(X,Y)].
\end{align}
where $X, X'$ are independent random variables drawn from $P$, and $Y, Y'$ are independent random variables drawn from $Q$. 

We adopt a special case of the three-way MMD balancing weights proposed in \cite{chen2024robust}.
Let $\mathcal{D}=\{(\x_i,a_i)\}_{i=1}^n$ denote the observed sample and let $\mathcal{S}:=\{1,\dots,n\}$
be its index set. Write $\mathcal{S}_0:=\{i\in\mathcal{S}:a_i=0\}$ and $\mathcal{S}_1:=\{i\in\mathcal{S}:a_i=1\}$
for the control and treatment index sets. We compute weights $\w(\mathcal{D})=(\w_0,\w_1)$ by minimizing a
linear combination of squared MMD terms among the three empirical covariate distributions
$\P_{\mathcal{S}_0}^{\w_0}$, $\P_{\mathcal{S}_1}^{\w_1}$, and $\P_{\mathcal{S}}$:
\begin{align}
	\w(\mathcal{D}) &:= \frac{1}{2}\argmin_{\w=(\w_{0},\w_{1})} \alpha \textup{MMD}^{2}_{\mathcal{H}}(\P^{\w_{0}}_{\mathcal{S}_{0}}, \P_{\mathcal{S}}) + \alpha \textup{MMD}^{2}_{\mathcal{H}}(\P^{\w_{1}}_{\mathcal{S}_{1}},\, \P_{\mathcal{S}}) \label{eq:MMD_weight_def} \\
	&\hspace{3cm} + (1-\alpha)  \textup{MMD}^{2}_{\mathcal{H}}(\P^{\w_{0}}_{\mathcal{S}_{0}},\,\P^{\w_{1}}_{\mathcal{S}_{1}} ) + \frac{\lambda_\textup{MMD}}{n^2} \| \w \|_{2}^{2} \nonumber \\
	&\textup{subject to $\sum_{i\in \mathcal{S}_{0}} \w_{0i}= \sum_{i\in \mathcal{S}_{1}} \w_{1i}=n$}. \nonumber
\end{align}
where $\P^{\w_{a}}_{\mathcal{S}_{a}}=\sum_{i\in \mathcal{S}_{a}} \frac{w_{ai}}{n}\delta_{\x_i}$ is the weighted empirical
distribution of covariates in group $a\in\{0,1\}$ and $\P_{\mathcal{S}}=\sum_{i\in \mathcal{S}}\frac{1}{n}\delta_{\x_i}$
is the empirical covariate distribution of the full sample, with fixed parameters $\alpha \in [0, 1]$ and $\lambda_\textup{MMD}\ge 0$. 

We note that the formulation in \eqref{eq:MMD_weight_def}, following \cite{chen2024robust}, we first solve for unnormalized weights $\widetilde{\w}=(\widetilde{\w}_0,\widetilde{\w}_1)$ satisfying the group-wise constraints
$\sum_{i\in \mathcal{S}_0} \widetilde{w}_{0i} = \sum_{i\in \mathcal{S}_1} \widetilde{w}_{1i} = 2n$. We then renormalize by a factor of $1/2$ and define
$\w(\mathcal{D}) := \frac{1}{2} \widetilde{\w}(\mathcal{D})$ so that the final weights satisfy the sum-to-$n$ normalization used in our DP-2ERM framework.

Following \cite{chen2024robust}, we can rewrite \eqref{eq:MMD_weight_def} as the following constrained quadratic minimization problem:
\begin{align}\label{eq:MMD_def}
	\w(\mathcal{D}) &=  \argmin_{ \w=(\w_{0},\w_{1}) ,\, \|\w \|_{\infty}\le R_\textup{MMD}} \,\, \left[ l(\w;\mathcal{D}):=\w^{\top} \A \w - 2 \w^{\top} \b \right],
\end{align}
where matrix $\A=\A(\mathcal{D},\alpha,\lambda_\textup{MMD})$ and vector $\b=\b(\mathcal{D},\alpha)$ are defined as
\begin{align}\label{eq:MMD_A_b_def}
	\A = \frac{1}{n^{2}}\quad 
	\begin{bNiceArray}{c I c}[first-row, first-col, cell-space-top-limit=3pt, extra-margin=4pt]
		&\small{a=0}              &\small{a = 1} \\
		\small{a=0}   &K_{\mathcal{S}_{0}, \mathcal{S}_{0}} + \lambda_\textup{MMD} I  & (\alpha-1) K_{\mathcal{S}_{0}, \mathcal{S}_{1}}   \\ \H
		\small{a=1}   &(\alpha-1) K_{\mathcal{S}_{1}, \mathcal{S}_{0}}   &K_{\mathcal{S}_{1}, \mathcal{S}_{1}} + \lambda_\textup{MMD} I
	\end{bNiceArray}, \quad
	\b = \frac{\alpha}{n^{2}} \quad 
	\begin{bNiceArray}{c}[first-row, first-col, cell-space-top-limit=3pt, extra-margin=4pt]
		\\
		a=0   &K_{\mathcal{S}_{0}, \mathcal{S}} \mathbf{1} \\ \H
		a=1   &K_{\mathcal{S}_{1}, \mathcal{S}} \mathbf{1}
	\end{bNiceArray},
\end{align}
where for $a, a' \in \{0, 1\}$, 
\begin{align}
	K_{\mathcal{S}_{a}, \mathcal{S}_{a'}} = (K(\x_{i}, \x_{j}))_{i \in \mathcal{S}_{a}, j \in \mathcal{S}_{a'}}, \quad  \textup{and}\quad  K_{\mathcal{S}_{a}, \mathcal{S}} = (K(\x_{i}, \x_{j}))_{i \in \mathcal{S}_{a}, j \in \mathcal{S}}.
\end{align}

If $R_\textup{MMD}>0$ is sufficiently large so that the sup-norm constraint in \eqref{eq:MMD_def} is not effective, then \eqref{eq:MMD_def} is equivalent to  \eqref{eq:MMD_weight_def} according to \cite[Lemma 6]{chen2024robust}. The following result provides a stability bound for MMD.

\begin{prop}[Stability Bound on MMD]\label{prop:MMD_stability}
	Let $\w$ and $\w'$ denote the MMD weights from datasets $\mathcal{D}$ and $\mathcal{D}'$ that differ at one point.
	Assume that $|K(\x_{i},\x_{j})|\le C$ for some constant $C>0$ for all $i,j$. Further assume the entries of
	$\w$ and $\w'$ satisfy $\|\w\|_\infty,\|\w'\|_\infty \le R_\textup{MMD}$. Then
	\begin{align}\label{eq:weight_bound}
		\| \w-\w' \|_{2} \le 2\sqrt{2}(R_\textup{MMD}+1)C n/\lambda_\textup{MMD}.
	\end{align}
\end{prop}

The stability bound for MMD highlights the crucial roles of the parameters $R_\textup{MMD}$ and $\lambda_\textup{MMD}$. The additional sup-norm constraint $\|\w\|_{\infty}\le R_\textup{MMD}$ is introduced to improve stability by preventing any single subject from receiving an extremely large weight.

The proposition also shows that the regularization parameter $\lambda_{\textup{MMD}}$ governs the sensitivity of the weights to data perturbation. When $\lambda_\textup{MMD}$ is $O(1)$, the bound grows linearly with the sample size $n$, so a small change in the data can induce a substantial shift in the weights. In contrast, when $\lambda_\textup{MMD}$ is of order $O(n^{2})$, the bound becomes $O(1/n)$.

A useful intuition for this weaker stability is that the MMD objective is built from kernel Gram matrices, so a single-record perturbation modifies an entire row and column of the coupled quadratic program rather than only a localized term. As a result, its effect propagates more globally through the optimization problem, making the solution more sensitive to perturbations unless sufficiently strong regularization is imposed.

This highlights an important trade-off: choosing a larger $\lambda_\textup{MMD}$ enhances stability but may compromise covariate balance, whereas a smaller $\lambda_\textup{MMD}$ improves balance at the cost of stability. In practice, the choice of $\lambda_\textup{MMD}$ should therefore balance these competing objectives in light of the data structure and the target application.

\subsubsection{Stability of EBW}

We first review the construction of EBW for ATE setting \cite{hainmueller2012entropy,zhao2017entropy,chen2023entropy}. 

Suppose we have a population $\mathcal{S}$ of $n$ patients consisting of $n_{0}$ patients in the control group $\mathcal{S}_{0}$ and $n_{1}$ patients in the treatment group $\mathcal{S}_{1}$. 

The goal is to reweight every patient so that the first $K$ moments in the reweighted control (resp., treatment) group match up with the corresponding moments for the entire population $\mathcal{S}$. Let $\w_{a}=(w_{a1},\dots,w_{ an_{a}})\in n_{a}\Delta^{n_{a}}$ for $a=0,1$ denote the weight vectors for the control and the treatment groups with normalization condition $\sum_{i\in \mathcal{S}_{a}} w_{ai} = n_{a}$ for $a=0,1$. Fix a base probability distribution for the entire group $\mathbf{q}=(q_{1},\dots,q_{n})$, which is often taken as the uniform distribution (i.e., $q_{i}\equiv 1/n$). We seek for weights $\w=(\w_{0},\w_{1})\in n_{0} \Delta^{n_{0}}\times n_{1}\Delta^{n_{1}}$ that is as close as possible to the base distribution $\mathbf{q}$ subject to the moment constraints with functionals $g_{0},\dots,g_{K}:\R^{p}\rightarrow \R$. This problem can be expressed as the following relative entropy minimization problem
\begin{align}\label{eq:entropy_weight_primal}
	\min_{\w \in n\Delta^{n}} \, D_{KL}(\w \Vert \mathbf{q}) \quad \textup{s.t.} \quad 
	&\frac{1}{n_{0}}\sum_{i\in \mathcal{S}_{0}}  w_{i} g_{r}(\x_{i}) = \frac{1}{n}\sum_{i\in \mathcal{S}}  g_{r}(\x_{i}), \\  
	& \frac{1}{n_{1}}\sum_{i\in \mathcal{S}_{1}}  w_{i} g_{r}(\x_{i}) = \frac{1}{n}\sum_{i\in \mathcal{S}}  g_{r}(\x_{i}), \quad \textup{for $r=0,\dots,K$}. 
\end{align}
Here $D_{KL}(\w \Vert \mathbf{q})=\sum_{i\in \mathcal{S}} (w_{i}/n) \log (w_{i}/nq_{i})$ denotes the KL divergence of $\w/n$ from $\mathbf{q}$, which is used as a measure of `distance' between $\w/n$ and $\mathbf{q}$. (Recall that $D_{KL}(\w \Vert \mathbf{q})=0$ if $\w/n=\mathbf{q}$ and equals $\infty$ if $\w/n$ is not absolutely continuous w.r.t. $\mathbf{q}$). In order to force group-wise normalization condition $\sum_{i\in \mathcal{S}_{a}} w_{i} = n_{a}$ for $a=0,1$, we assume $g_{0}(\cdot)\equiv 1$, the zeroth moment function. 
Once we find an optimal weight vector $\hat{\w}\in n\Delta^{n}$, this gives the weights for our 2ERM formulation in \eqref{eq:weighted_convex_opt}.
We assume throughout that the positivity condition holds. However, positivity alone does not guarantee the feasibility of the exact moment-matching constraints in \eqref{eq:entropy_weight_primal}: if the number of balance functions $K$ is large relative to the within-group sample sizes $n_0$ and $n_1$, the constraints may become overdetermined even under adequate overlap. As discussed below, the dual formulation of \eqref{eq:entropy_weight_primal} admits a natural regularized relaxation that replaces exact matching with approximate balance, thereby yielding a more stable and practically feasible formulation. In practice, infeasibility may also be alleviated by reducing the number or complexity of the balance functions \cite{wang2020minimal}.

Following the approach in Hainmueller \cite[Section 3.2]{hainmueller2012entropy}, we can find the optimal weights $\w^{*}$ in the relative entropy minimization problem \eqref{eq:entropy_weight_primal} by solving its dual formulation. 

Denote $\mathbf{g}=(g_{0},\dots,g_{K}):\R^{p}\rightarrow \R^{K+1}$, which is a function that takes a $p$-dimensional covariate vector $\x_{i}$ into a $(K+1)$-dimensional summary vector $(1,g_{1}(\x_{i}),\dots,g_{K}(\x_{i}))$. Also denote 
\begin{align}\label{eq:def_g_bar_entropy}
	\overline{\mathbf{g}} := \frac{1}{n}\sum_{i=1}^{n}  \mathbf{g}(\x_{i}), 
\end{align}
as the mean moment vector for the entire population $\mathcal{S}$. 
By using first principles, we will show that the dual problem for \eqref{eq:entropy_weight_primal} is
\begin{align}\label{eq:entropy_dual}
	\max_{\blambda=(\blambda_{0},\blambda_{1}),\, \|\blambda \|_{\infty} \le R_\textup{EBW}} \left[ \mathcal{E}^{d}(\blambda_{0},\blambda_{1}) := \left\langle \blambda_{0}+\blambda_{1} ,\, \overline{\mathbf{g}} \right\rangle  - \log C(\blambda) - \frac{\lambda_{\textup{EBW}}}{2} \| \blambda \|_{2}^{2}   \right], 
\end{align}
where 
\begin{align}\label{eq:def_C_lambda}
	C(\blambda) 
	&=:  \sum_{i=1}^{n} q_{i} \exp\left(  \frac{n}{n_{a_{i}}}\left\langle \blambda_{a_{i}}, \mathbf{g}(\x_{i})  \right\rangle \right).
\end{align}
provided $R_\textup{EBW}>0$ is sufficiently large and the $L_{2}$-regularization coefficient $\lambda_{\textup{EBW}}$ is zero. Note that adding $\lambda_{\textup{EBW}} > 0$ effectively relaxes the strict moment-matching constraints in \eqref{eq:entropy_weight_primal} to approximate matching, aligning with the framework of minimal weights \cite{wang2020minimal}. This enhances numerical stability and is essential for the privacy stability bounds derived in Theorem \ref{thm:entropy_weight_stability}. Also note that \eqref{eq:entropy_dual} is a strictly concave maximization problem, which can be solved by a variety of convex optimization algorithms. Once we find the optimal dual variable $\blambda^{*}$ for \eqref{eq:entropy_dual}, we can compute the optimal weights $\w^{*}$ as follows. 

\begin{prop}\label{prop:entropy_weight_duality}
	Let $\blambda^{*}=(\blambda_{0}^{*},\blambda_{1}^{*})$ be the optimal solution for \eqref{eq:entropy_dual}. Then the solution $\w^{*}=(w_{1}^{*},\dots,w_{n}^{*})$ to \eqref{eq:entropy_weight_primal} is given by 
	\begin{align}\label{eq:entropy_u_lambda}
		w_{i}^{*} = 
		\frac{n}{C} q_{i} \exp \left(\frac{n}{n_{a_{i}}}\left\langle \blambda_{a_{i}}^{*}, \mathbf{g}(\x_{i})  \right\rangle \right), 
	\end{align}
	where $C=C(\blambda^{*})$ is the sum-to-$n$ normalization constant and it equals to \eqref{eq:def_C_lambda} with $\blambda=\blambda^{*}$. 
\end{prop}

Now we state our stability result for EBW. Since the primal problem \eqref{eq:entropy_weight_primal} is scale-invariant under the transformation $g(\x_{i})\mapsto a g(\x_{i})$ for any scalar $a>0$, we may assume without loss of generality that $\frac{n}{n_{0}\land n_{1}}\| g(\x_{i}) \|_{2}\le 1$ for all $i$.

\begin{theorem}[Stability of EBW]\label{thm:entropy_weight_stability}
	Let $\w$ and $\w'$ denote the EBW weights for datasets $\mathcal{D}$ and $\mathcal{D}'$ that differ at one point, as above, with dual variables restricted to the $L_2$-ball $\|\blambda\|_{2}\le R_\textup{EBW}$, under the observational study.  
	Let $\B \in \R^{n\times 2(K+1)}$ denote the extended covariate matrix for $\mathcal{D}$ where each row $\b_{i}$ is defined as
	\begin{align}\label{eq:def_entropy_b}
		\b_{i}:=
		\frac{n}{n_{0}}
		\begin{pmatrix}
			\mathbf{g}(\x_{i}) \\
			\mathbf{0}
		\end{pmatrix}
		\quad \textup{for $a_{i}=0$}, \qquad  
		\b_{i}:=
		\frac{n}{n_{1}}
		\begin{pmatrix}
			\mathbf{0} \\
			\mathbf{g}(\x_{i})
		\end{pmatrix}
		\quad \textup{for $a_{i}=1$}.
	\end{align}
	Suppose without loss of generality that $\|\b_{i}\|_{2}\le 1$ for all $i$.  Let $\hat{\Sigma}\in \R^{2(K+1) \times 2(K+1)}$ denote the empirical covariance matrix of $\B$. Define the base-measure ratio $r_{q}:=\frac{\max_{i} q_{i}}{ \min_{i} q_{i} } \ge 1$ and the strong concavity parameter
	\begin{align}
		\rho := r_{q}^{-1} \exp(-2R_\textup{EBW}) \lambda_{\min}(\hat{\Sigma}) + \lambda_{\textup{EBW}} > 0, \qquad \delta_{E} := \left( R_\textup{EBW} \wedge  \frac{1+\sqrt{2}}{\rho}\right).
	\end{align}
	Then, 
	\begin{align}\label{eq:entropy_sensitivity_bd}
		\|\w - \w'\|_2 &\le \frac{2\,r_q\exp(2R_\textup{EBW})\bigl(\sqrt{2}+r_q\exp(2R_\textup{EBW})\bigr)}{\rho\sqrt{n}} + \frac{2R_\textup{EBW}\,r_q^2\exp(4R_\textup{EBW})}{\rho} \\
		& \qquad + 2\sqrt{2} \, r_q \delta_{E}\, \exp(2\delta_{E}). \nonumber
	\end{align}
\end{theorem}

The three terms in~\eqref{eq:entropy_sensitivity_bd} reflect distinct
mechanisms. The first term, of order $O(1/\sqrt{n})$, captures the
parameter shift $\|\hat{\blambda}-\hat{\blambda}'\|_2$ propagated
through the Lipschitz constant of the weighted softmax. Its
$1/\sqrt{n}$ decay---absent in the analogous IPW bound
(Proposition~\ref{prop:IPW_estimated_stability}), whose leading term is
$O(1)$---arises because the softmax is $O(1/n)$-Lipschitz rather than
$O(1/\sqrt{n})$-Lipschitz, and is the key structural advantage of EBW
over IPW. The remaining two terms govern the data-shift contribution
and are controlled by $1/\rho$, so that stronger regularization
$\lambda_{\textup{EBW}}$ renders them negligible at the cost of
relaxing exact moment-matching. Crucially, the EBW regularization acts
through the dual parameter $\blambda$ as a nuisance quantity, whereas
IPW's regularization $\lambda_{\textup{IPW}}$ directly perturbs the
released propensity model; consequently, EBW can leverage large
$\lambda_{\textup{EBW}}$ to reduce sensitivity without the
misspecification cost that a large $\lambda_{\textup{IPW}}$ would
incur. Finally, the dependence on the dual radius through the tighter
quantity $\delta_{E}=R_\textup{EBW}\wedge(1+\sqrt{2})/\rho$ ensures that both terms
also vanish in the small-$R_\textup{EBW}$ regime, even when $\rho$ is moderate.

We verify in Appendix that the $R=O(1)$ regime of Corollary~\ref{cor:DP_objective_perturbation} holds by construction or under mild overlap for each of the three weighting methods.

\section{Learning ITRs via Weighted Regression and DP-2ERM}\label{sec:ITR}

As discussed in Section~\ref{sec:running_example_ITR}, learning ITRs from observational data is a canonical instance of 2ERM: covariate balancing weights are computed in a first stage to mitigate confounding, and then treated as nuisance inputs in a second-stage weighted ERM. In this section, we formalize the weighted-regression reduction for ITR learning and instantiate Algorithm~\ref{algorithm:main} for private estimation of a linear decision function.

\subsection{Setup and Optimal Decision Function}
We observe i.i.d.\ data $(\X,A,Y)$ with $A\in\{0,1\}$. We learn a decision function $f$ and use
$d(\x)=\mathbf{1}\{f(\x)>0\}$. Throughout, the treatment is encoded as $A\in\{0,1\}$; the expression $(2A-1)\in\{-1,+1\}$ that appears inside the weighted-regression loss below is an algebraic sign transform of this $\{0,1\}$-encoded $A$, not a re-encoding of the treatment. Under the causal assumptions in Section~\ref{sec:running_example_ITR}, the value
$V(d):=\E\{Y(d(\X))\}$ is identified by
\begin{align}
	V(d)=\E\!\left[\frac{Y\,\mathbf{1}\{A=d(\X)\}}{\pi(A,\X)}\right].    
\end{align}
The optimal ITR is determined by the conditional mean contrast
\begin{align}
	f_{\textup{opt}}(\x):=\E(Y\mid \X=\x,A=1)-\E(Y\mid \X=\x,A=0),
	\,\,
	d_{\textup{opt}}(\x)=\mathbf{1}\{f_{\textup{opt}}(\x)>0\}.
\end{align}

\subsection{Weighted Regression Reduction, 2ERM Structure, and Private Estimation}\label{sec:ITR_weighted_regression}
A key observation underlying the \emph{weighted regression} framework for ITR learning
\cite{tian2014simple,chen2017general,qi2018d} is that the optimal decision function
$f_{\textup{opt}}$ is characterized as the minimizer of a weighted square-loss risk,
\begin{align*}
	f_{\textup{opt}} \in \argmin_{f}\,\, \E\!\left[\frac{(2Y(2A-1)-f(\X))^{2}}{\pi(A,\X)}\right].
\end{align*}

\noindent\textit{Intuition.}
Weighting corrects for the observational treatment mechanism by reweighting the data to a
pseudo-population in which the treated and control covariate distributions are balanced; in a
randomized trial this is automatic since $\pi(a,\x)\equiv 1/2$. In this balanced population,
regressing the transformed outcome $2Y(2A-1)$ on $\X$ targets the conditional mean contrast without
requiring a parametric model for $\E(Y\mid A,\X)$.

Formally, fix $\x$ and write $g=f(\x)$. Then
\begin{align}
	\E\!\left[\left.\frac{(2Y(2A-1)-g)^2}{\pi(A,\X)}\right|\,\X=\x\right]
	=\sum_{a\in\{0,1\}}\E\!\left[(2(2a-1)Y-g)^2 \mid \X=\x, A=a\right],
\end{align}
where $1/\pi(A,\X)$ cancels $\P(A=a\mid \X=\x)$. Differentiating with respect to $g$ yields the unique
minimizer $g=f_{\textup{opt}}(\x)$. Consequently, $f_{\textup{opt}}$ can be estimated by a weighted
least-squares regression of $2Y(2A-1)$ on covariates (or interaction features), and the induced ITR is
$d(\x)=\mathbf{1}\{f(\x)>0\}$.

We adopt the linear class $f_{\param}(\x)=\x^\top\param$ for interpretability and because it yields a convex
weighted ERM amenable to private optimization. Given $\mathcal{D}=\{(\x_i,a_i,y_i)\}_{i=1}^n$, we estimate
$\param$ by the weighted ERM
\begin{align}\label{eq:weighted_convex_opt_itr}
	\widehat{\param}\in\argmin_{\param\in \Param}
	\left\{
	\frac{1}{n}\sum_{i=1}^n w_i\, \big(2y_i (2a_i-1)-\x_i^\top\param\big)^2 + R(\param)
	\right\},
	\qquad
	\Param:=\{\param\in\R^p:\|\param\|_1\le \lambda_1\},
\end{align}
where $R(\param)$ is a convex regularizer.

When $\pi$ is known, $w_i=1/\pi(a_i,\x_i)$ recovers the usual inverse-propensity weighted criterion
\cite{tian2014simple,chen2017general}. More generally, we allow data-dependent weights
$\w(\mathcal{D})=(w_1,\dots,w_n)$, and in particular the covariate balancing weights from
Section~\ref{sec:stability_CBW}, which enforces covariate balance without explicitly modeling $\pi$.

Equation~\eqref{eq:weighted_convex_opt_itr} is an instance of the 2ERM 
template~\eqref{eq:weighted_convex_opt}:
the first stage computes weights $\w(\mathcal{D})$ by 
solving~\eqref{eq:find_optimal_weight}, and the second stage solves
\eqref{eq:weighted_convex_opt_itr} conditional on $\w(\mathcal{D})$. Consequently, the 
privacy noise level and the additional $L_2$-regularization in 
Algorithm~\ref{algorithm:main} are calibrated through the weight
stability $\|\w(\mathcal{D})-\w(\mathcal{D}')\|$ via 
Theorem~\ref{thm:sensitivity_objective_perturbation} (see also 
Section~\ref{sec:running_example_ITR}); tighter stability bounds
permit smaller objective perturbations and improved utility. Following 
Lee et al.~\cite{lee2024effective}, we
solve~\eqref{eq:weighted_convex_opt_itr} using \textit{projected gradient descent} (PGD).

In the following corollary, we invoke 
Theorem~\ref{thm:sensitivity_objective_perturbation} to obtain a differentially 
private estimator of $\widehat{\param}$ via objective perturbation.

\begin{corollary}[Data Sensitivity Bound for DP-ITR via Objective Perturbation]
	\label{cor:sensitivity_objective_perturbation_ITR}
	Suppose $\mathcal{D}$ and $\mathcal{D}'$ are two size-$n$ datasets that differ in exactly one 
	observation, with corresponding (non-private) optimal weights 
	$\w(\mathcal{D})=(w_{1},\dots,w_{n})$ and $\w(\mathcal{D}')=(w_{1}',\dots,w_{n}')$ solving 
	\eqref{eq:find_optimal_weight}. Suppose the parameter space is 
	$\Param=\{\param:\|\param\|_{1}\le \lambda_{1}\}$ for some $\lambda_{1}>0$. Let 
	$\param^{\textup{priv}}$ be the output of Algorithm~\ref{algorithm:main} applied to 
	\eqref{eq:weighted_convex_opt_itr}. If $\|\x\|_2\le M$ and $|y|\le M'$ for all 
	$(\x,a,y)\in\mathcal{D}\cup\mathcal{D}'$, then \eqref{eq:thmA1_gamma} and 
	\eqref{eq:thmA1_gaussian} in Theorem~\ref{thm:sensitivity_objective_perturbation} hold 
	with $\zeta= 2M^{2}\lambda_{1} + 4MM'$ and $\lambda=2M^{2}$.
\end{corollary}

\section{Simulation Studies}\label{sec:simulation}
To evaluate the performance of our DP-2ERM approach, we conduct simulation studies under an observational study setting, adapting the data-generating mechanisms in \cite{lee2024effective, shah2022stabilized, maronge2023reluctant}. In each replicate, we generate $n=400$ training observations and evaluate the learned ITR on an independent test set of size $n_{\textup{test}}=10{,}000$. Covariates are $\X=(X_1,\ldots,X_p)$ with $p=10$ and $X_j \stackrel{\textup{iid}}{\sim} \textup{TN}_{[-1,1]}(0,1)$, i.e., a standard normal truncated to $[-1,1]$. We repeat the experiment for 100 replicates.

\noindent{\it Treatment assignment and outcome generation}: Across all scenarios, treatment $A\in\{0,1\}$ follows $\textup{logit}\{\P(A=1\mid \X)\}=0.3X_1-0.5X_2+0.05$, and outcomes are generated from $Y=\mu(\X)+\frac{2A-1}{2}\,f_{\textup{opt}}(\X)+\tilde{\varepsilon}$, where $\mu(\X)$ is the treatment-free effect and $\tilde{\varepsilon}$ is mean-zero noise with conditional variance
$\Var(\tilde{\varepsilon}\mid A,\X)=\sigma^2(A,\X)$.

\noindent{\it Scenario-specific mechanisms}: We consider three settings: one correctly specified linear scenario and two misspecified scenarios (tree and nonlinear). Each scenario specifies $(\mu(\X), f_{\textup{opt}}(\X), \sigma^2(A,\X))$ as follows:

\begin{enumerate}[label=(\roman*), leftmargin=*, itemsep=6pt]
	\item \textbf{Linear optimal ITR (correctly specified):}
	\begin{align*}
		&\mu(\X) = -0.1 \sum_{j=1}^{5}\left\{X_j+\frac{2}{3}(2X_j^2-1)\right\}, \quad  f_{\textup{opt}}(\X) = 8X_1-8X_2+4X_3+8X_4, \\& \sigma^2(A,\X) = 2.
	\end{align*}
	
	\item \textbf{Tree optimal ITR (misspecified):}
	\begin{align*}
		&\mu(\X) = -3 \sum_{j=1}^{5}\left\{X_j+\frac{2}{3}(2X_j^2-1)\right\}, \quad \sigma^2(A,\X) = 2, \\
		&f_{\textup{opt}}(\X) =
		6\,\mathbb{I}(X_1>-0.5)\,\textup{sign}(X_1-0.5)
		\;+\;
		5\,\mathbb{I}(2X_1<-0.5)\,\textup{sign}(X_4+0.5)
		\;+\; 1.
	\end{align*}
	
	\item \textbf{Nonlinear optimal ITR (misspecified; heteroscedastic).}
	\begin{align*}
		\mu(\X) &= 2 + 3X_1 + 2X_2 + 3X_4 - 2.5X_4^2 - 1.5X_5^2
		+ 2X_1X_2 + 2\exp(-X_1X_2) + \sin(X_3),\\
		f_{\textup{opt}}(\X) &= -0.5 - 2X_4 + X_4^2 + 2.5X_5^2,\\
		\sigma^2(A,\X) &= 0.25 + 2X_2\,\mathbb{I}(X_2>0)
		+ X_3\,\mathbb{I}(X_3>0,\,A=1)
		+ X_4\,\mathbb{I}(X_4>0,\,A=0).
	\end{align*}
\end{enumerate}

We compare DP-2ERM against composition-style baselines tailored to each weighting family. The construction differs between IPW and EBW/MMD:
\begin{itemize}[leftmargin=*]
	\item \textbf{\texttt{IPW (Private - composition)}}: a literal stage-wise composition. The first stage privately fits the logistic propensity model with budget $\varepsilon/2$ via the Chaudhuri--Monteleoni--Sarwate mechanism \cite{chaudhuri2011differentially}; the second stage runs the data-independent-weight DP-ERM of Corollary~\ref{cor:DP_guarnatee_weighted_ERM} calibrated using the deterministic bound $\max_i \tilde{w}_i \le 1 + \exp(M R_{\textup{IPW}})$, which holds under $\|\hat{\lambda}\|_2 \le R_{\textup{IPW}}$.
	\item \textbf{\texttt{EBW/MMD (Private - composition)}}: no off-the-shelf DP first-stage mechanism is available, so only the second stage is privatized using the method-specific regime of Corollary~\ref{cor:DP_objective_perturbation}. Each family's first-stage constraint yields $\max_i w_i \le R_{\textup{method}}$, which sharpens the unconditional $\overline{W}_2 = O(n^{3/2})$ to $\overline{W}_2 = O(n\sqrt{R_{\textup{method}}})$ and tightens $\overline{W}_1$ from $4n$ to $\le 2n + 2R_{\textup{method}}$.
\end{itemize}

Hyperparameters $(R_{\textup{method}}, \lambda_{\textup{method}})$ are jointly tuned on an independently simulated validation set disjoint from every training replicate---so tuning consumes no privacy budget on the training data---by maximizing average bootstrap-\textit{out-of-bag} (OOB) accuracy across privacy budgets, following \cite{wang2020minimal,chen2023robust}.

Performance is summarized by \emph{accuracy}, defined as the fraction of test points for which the estimated ITR matches the true optimal treatment, evaluated across a range of privacy budgets $\varepsilon$. For Gaussian objective perturbation we set $\delta = 1/n$ (or $\delta = 1/(2n)$ per stage for \texttt{IPW (Private - composition)}, so the overall Gaussian guarantee is $(\varepsilon, 1/n)$-DP), a standard choice in the differential privacy literature.

\begin{figure}[tbp]
	\centering
	\includegraphics[width=\textwidth]{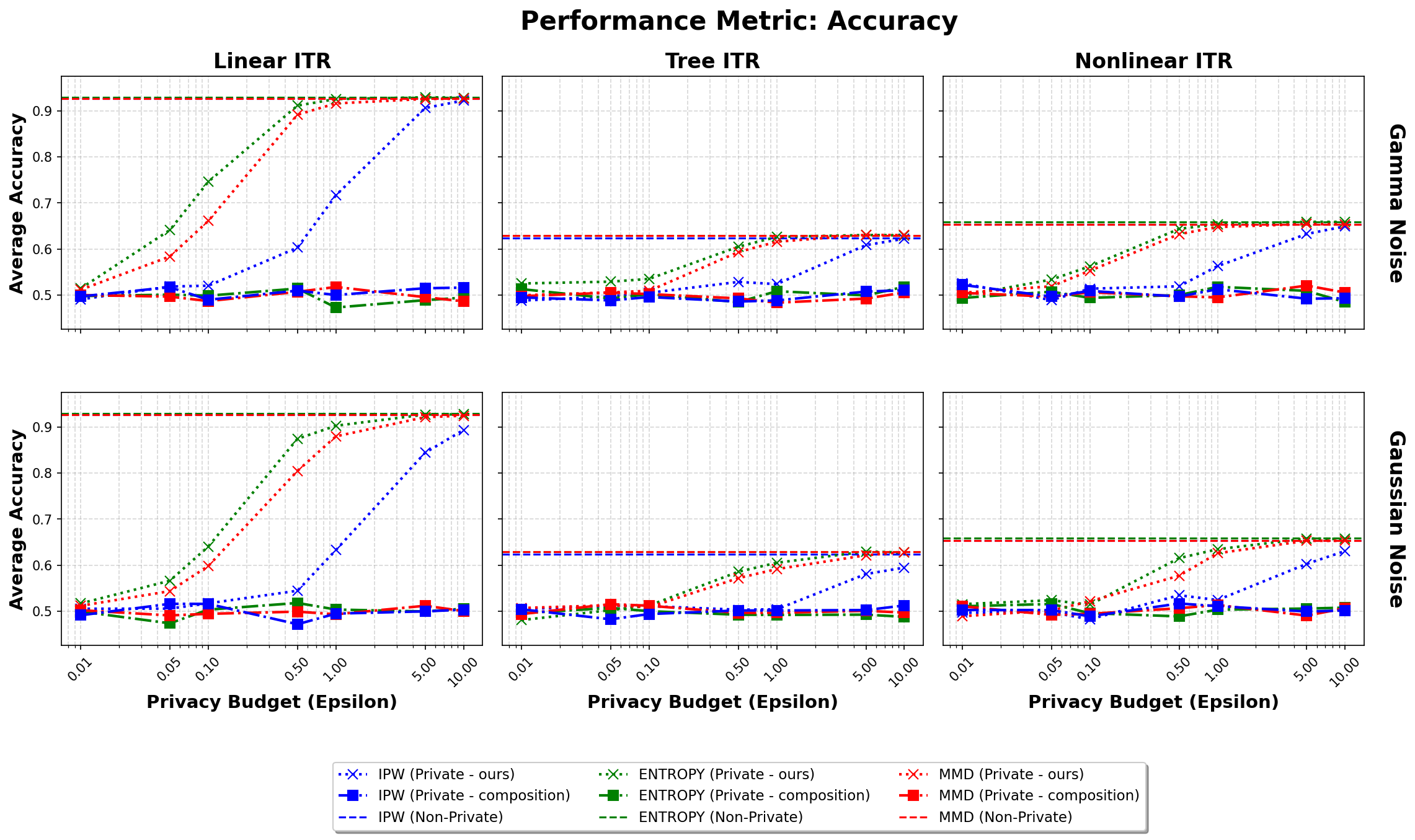}
	\caption{Privacy--utility trade-off across scenarios. Average test-set ITR accuracy 
		(100 replicates) for DP-2ERM with IPW (blue), EBW (green), and MMD (red) under 
		Gamma objective perturbation (top; $\varepsilon$-DP) and Gaussian objective perturbation 
		(bottom; $(\varepsilon,\delta)$-DP). Solid squares denote private composition-style 
		baselines (\texttt{Private-composition}): for IPW, a literal stage-wise composition 
		baseline with a private first-stage logistic propensity fit and a private second-stage 
		weighted ERM; for EBW and MMD, favorable worst-case second-stage surrogates, since 
		no off-the-shelf private first-stage mechanism is available for these weighting 
		programs. Dashed lines denote the non-private linear-ITR baseline within the chosen 
		decision-function class.}
	\label{fig:simulation}
	\vspace{-0.1in}
\end{figure}

Figure~\ref{fig:simulation} exhibits the privacy--utility pattern across the three scenarios we study. As $\varepsilon$ increases, DP-2ERM estimators approach their non-private counterparts; in the high-privacy regime (small $\varepsilon$), stronger privacy protection requires larger perturbations and accuracy degrades. Our theory makes this trade-off explicit by quantifying the amount of noise required at each privacy level for each weighting method.

The most striking empirical pattern is the gap between DP-2ERM (\texttt{Private-ours}) and \emph{all three} composition baselines. Across every scenario, every noise mechanism, and every privacy budget we report, the composition baselines hover near chance accuracy ($\approx 0.50$), while DP-2ERM approaches its non-private benchmark as $\varepsilon$ grows. Two structural sources drive this gap. (i) \textit{Budget splitting.} Literal composition must allocate $\varepsilon$ across the two stages; a $50/50$ split leaves the second stage with $\varepsilon/2$, contributing a factor of $2$ to the injected noise. (ii) \textit{Weight-magnitude versus weight-perturbation sensitivity.} Composition's second-stage sensitivity scales with $\max_i \tilde{w}_i$---the \emph{magnitude} of a single weight---because it must protect against adversarially concentrated weight vectors. DP-2ERM instead uses the weight-perturbation object $\lVert w(\mathcal{D}) - w(\mathcal{D}')\rVert$, whose bound for each weighting method carries a denominator $\rho$ that grows linearly in the first-stage regularizer: for IPW with estimated logistic parameters, $\rho = e^{-MR}(1+e^{-MR})^{-2}\lambda_{\min}(\hat{\Sigma}) + \lambda_{\textup{IPW}}$ (Proposition~\ref{prop:IPW_estimated_stability}), with analogous structure for EBW. Increasing the first-stage regularizer therefore shrinks DP-2ERM's sensitivity linearly across \emph{all} weighting methods, whereas the composition baseline's max-weight bound is insensitive to this regularization. In the Linear scenario these two effects combine to an approximately $117$ times larger second-stage noise scale for \texttt{IPW (Private - composition)} than for DP-2ERM at the same $\varepsilon$. For EBW and MMD, the gap is even larger because the worst-case surrogate pays $\overline{W}_1 \le 2n + 2R_{\textup{method}} = O(n)$, which cannot be sharpened without first-stage stability information.

Comparing noise mechanisms, the Gamma perturbation (top row, $\varepsilon$-DP) consistently yields higher accuracy than the Gaussian perturbation (bottom row, $(\varepsilon,1/n)$-DP) in the high- and moderate-privacy regimes (e.g., in the Linear scenario at $\varepsilon=0.50$, EBW--Gamma attains $0.91$ versus $0.87$ for EBW--Gaussian). As $\varepsilon$ increases the two mechanisms become more similar and both approach the corresponding non-private baselines, with Gamma typically reaching the plateau earlier.

Among the weighting methods, EBW (green) and MMD (red) exhibit comparably fast convergence toward their non-private baselines, while IPW (blue) lags noticeably, particularly in the intermediate-privacy regime. IPW's slower recovery reflects its higher sensitivity to perturbations of the estimated propensity model.

Finally, DP-2ERM remains effective under misspecification of the decision-function class. In the Tree and Nonlinear scenarios, the best achievable performance within the linear-ITR class (dashed lines) is only moderately above chance ($\approx 0.63$ and $0.66$, respectively), reflecting the inherent approximation error of a linear rule. Importantly, DP-2ERM with stable weights converges to these non-private linear baselines as privacy constraints relax, confirming that in these simulated settings the method recovers the best rule in the chosen decision-function class even when the true optimal rule is nonlinear. Additional results on empirical values (i.e., the estimated value function evaluated on the test set) are deferred to Appendix.

\section{Applications}\label{sec:application}

To evaluate the performance of the proposed algorithm DP-2ERM on a standard benchmark built from real covariates, we use the well-known Twins dataset, a common benchmark for evaluating CATE estimators in machine learning \cite{almond2005costs, crabbe2022benchmarking}. This dataset is derived from vital statistics regarding twin births in the United States. To create a challenging benchmark that mimics an observational study, we employ the semi-synthetic data generating process outlined in Algorithm 1 of Crabbé et al. \cite{crabbe2022benchmarking}. This approach retains the realistic correlation structure of the original covariates while synthetically generating treatment assignments and outcomes. The algorithm explicitly selects non-overlapping subsets of prognostic covariates, which determine the outcome regardless of treatment, and predictive covariates, which determine the outcome per treatment. By using this generation process, the underlying conditional average treatment effect is known, allowing for a precise evaluation of the estimated rules.

Following the protocol detailed in Appendix, we partition the data into a disjoint $20\%$ tuning pool and $80\%$ evaluation pool, fix hyperparameters $(R_{\textup{method}}, \lambda_{\textup{method}})$ once on the tuning pool, and draw $100$ training subsets of size $1{,}140$ from the evaluation pool; each released estimator is individually $\varepsilon$-DP (resp.\ $(\varepsilon, \delta)$-DP) with respect to the evaluation pool. We follow the same evaluation protocol (privacy mechanisms, privacy budgets, weighting methods) as in Section~\ref{sec:simulation}, with accuracy computed against the optimal rule implied by the known semi-synthetic CATE. The reported DP guarantees are with respect to the evaluation pool; the tuning pool is treated as auxiliary data outside the formal guarantee. See Appendix for a discussion of preprocessing and a fully private alternative.

Table~\ref{tab:application} summarizes accuracy across privacy budgets. All methods improve as $\varepsilon$ increases, and the Gamma mechanism generally outperforms the Gaussian mechanism across methods and budgets, with rare exceptions at the tightest privacy budgets. At the tightest budget ($\varepsilon=0.01$), Gamma-based accuracies are broadly similar across methods (EBW $0.66$, MMD $0.73$, IPW $0.72$); notably, both IPW--Gamma and MMD--Gamma sit close to the all-control baseline ($0.71$), indicating that under severe noise, the learned rule is close to trivial. As $\varepsilon$ grows, EBW--Gamma and MMD--Gamma recover more strongly, each reaching $0.91$--$0.92$ at $\varepsilon=0.50$: EBW--Gamma effectively matches its non-private benchmark ($0.91$), while MMD--Gamma also stabilizes at $0.91$. IPW--Gamma improves more gradually and plateaus near $0.86$ by $\varepsilon=0.50$, essentially matching its own non-private benchmark ($0.86$). Variability across repetitions decreases as $\varepsilon$ increases, consistent with reduced perturbation at larger privacy budgets. Empirical-value results show a similar pattern; for example, at $\varepsilon=0.10$, EBW--Gamma achieves empirical value $3.79$, compared with $3.79$ for MMD--Gamma and $3.73$ for IPW--Gamma, all three exceeding the all-control baseline ($3.45$). Complete empirical-value results are reported in Appendix.

{\small
	\begin{table}[tbp]
		\centering
		\caption{Average accuracy and standard deviations (in parentheses) of DP-2ERM across privacy budgets ($\varepsilon$) on the Twins dataset. The non-private accuracy is shown in the final column. Hyperparameters are tuned once on a disjoint 20\% holdout pool, so each reported estimator is individually $\varepsilon$-DP or $(\varepsilon, \delta)$-DP with respect to the 80\% evaluation pool. The accuracies for the baselines---assigning all to treatment and assigning all to control---are 0.29 (0.00) and 0.71 (0.00), respectively.}
		\label{tab:application}
		\begin{tabular}{@{}lrrrrrrrr@{}}
			\hline
			& \multicolumn{8}{c}{Privacy Budget ($\varepsilon$)} \\
			\cline{2-9}
			Method &
			\multicolumn{1}{c}{0.01} &
			\multicolumn{1}{c}{0.05} &
			\multicolumn{1}{c}{0.10} &
			\multicolumn{1}{c}{0.50} &
			\multicolumn{1}{c}{1.00} &
			\multicolumn{1}{c}{5.00} &
			\multicolumn{1}{c}{10.00} &
			\multicolumn{1}{c@{}}{Non-Private} \\
			\hline
			{EBW } - Gamma& 0.67 (0.12) & 0.81 (0.08) & 0.87 (0.06) & 0.91 (0.04) & 0.91 (0.03) & 0.91 (0.03) & 0.91 (0.03) & 0.92 (0.03) \\
			{EBW} - Gaussian& 0.61 (0.14) & 0.69 (0.11) & 0.73 (0.10) & 0.82 (0.08) & 0.84 (0.05) & 0.89 (0.04) & 0.90 (0.04) & 0.92 (0.03) \\
			{IPW } - Gamma& 0.74 (0.09) & 0.84 (0.05) & 0.85 (0.03) & 0.86 (0.03) & 0.86 (0.03) & 0.86 (0.03) & 0.86 (0.03) & 0.90 (0.03) \\
			{IPW} - Gaussian& 0.66 (0.11) & 0.72 (0.09) & 0.76 (0.07) & 0.84 (0.05) & 0.85 (0.04) & 0.86 (0.04) & 0.86 (0.04) & 0.90 (0.03) \\
			{MMD } - Gamma& 0.74 (0.09) & 0.85 (0.04) & 0.86 (0.03) & 0.87 (0.03) & 0.87 (0.03) & 0.86 (0.03) & 0.86 (0.03) & 0.97 (0.01) \\
			{MMD} - Gaussian& 0.67 (0.10) & 0.75 (0.09) & 0.77 (0.07) & 0.84 (0.04) & 0.85 (0.04) & 0.86 (0.03) & 0.86 (0.03) & 0.97 (0.01) \\
			\hline
		\end{tabular}
	\end{table}
}

\section{Discussion}
This paper develops a framework for end-to-end differentially private estimation within two-stage architectures where a data-dependent first stage---designed to enforce distributional constraints such as covariate balance---feeds into a second-stage weighted empirical risk minimization. The framework rests on two main results: a sensitivity bound for objective perturbation under data-dependent reweighting and non-asymptotic stability bounds for the first-stage weight maps of three standard covariate-balancing schemes. Together, they avoid the conservatism of generic composition-based calibration. A practical implication is that the privacy--utility trade-off is driven not only by the privacy mechanism but also by the \emph{stability} of the weighting scheme: our theory and experiments suggest that optimization-based balancing weights can be substantially more stable than IPW, which translates into markedly improved private performance at the same privacy budget.

Our theory relies on conditions typical for objective perturbation—convexity of the loss, and strong convexity induced by regularization—together with boundedness/regularity assumptions used to control sensitivity. Extending DP-2ERM to nonconvex objectives (e.g., deep policy networks) will require different private optimization tools and sensitivity analyses, such as DP-SGD \cite{xia2025statistical} or output perturbation under alternative stability conditions. Moreover, limited overlap can amplify instability of weights and therefore increase the required noise; understanding the precise interaction between overlap diagnostics, weight trimming, and privacy calibration is an important direction for further study.

Several extensions appear promising. First, DP-2ERM can be adapted to richer decision-function classes beyond linear rules and to multi-stage decision problems such as dynamic treatment regimes \cite{kosorok2019precision}. Second, the same analysis blueprint applies to other 2ERM-style pipelines, including covariate-shift correction and transfer learning with balancing weights. Finally, it would be useful to develop private inference and uncertainty quantification for the learned rule \cite{cai2023private}.

\bibliographystyle{imsart-number}   
\bibliography{mybib}


\newpage
\begin{appendix}

\pagenumbering{arabic}
\setcounter{page}{1}

\section{Proofs of Key Lemmas and Main Results}

Omitted for the initial arXiv submission. 

\section{Proof of Stability of Covariate Balancing Weights}\label{sec:proof_stability_CBW}

Omitted for the initial arXiv submission.

\end{appendix}

\end{document}